\documentclass[a4paper,10pt]{article}
\usepackage[utf8x]{inputenc}
\usepackage{tracefnt,amsmath,tabu,array}
\usepackage{amssymb,graphicx,setspace,amsfonts,amsbsy}
\usepackage{pifont,latexsym,ifthen,amsthm,rotating,calc,textcase,booktabs,cancel,slashed}

\newtheorem{theorem}{Theorem}[section]
\newtheorem{lemma}[theorem]{Lemma}

\newtheorem{proposition}[theorem]{Proposition}

\newtheorem{remark}[theorem]{Remark}
\newcommand{\filledbox}{\leavevmode
  \hbox to.77778em{%
  \hfil\vbox to.675em{\hrule width.6em height.6em}\hfil}}

\newcommand{\Rm}{{\mathbb R}}

\begin{document}
\tabulinesep=1.0mm
\title{Nonexistence of pure bubble type II blow-up solutions to energy critical wave equation in the 3D radial case}

\author{Ruipeng Shen\\
Centre for Applied Mathematics\\
Tianjin University\\
Tianjin, China
}

\maketitle

\begin{abstract}
 In this work we consider the focusing, energy-critical wave equation in 3D radial case. According to the soliton resolution conjecture, which has been verified in the radial case, any type II blow-up solution decomposes into a superposition of several decoupled grounds states, a free wave and a small error, as time tends to the blow-up time. We prove that there does not exist any pure bubble type II blow up solutions. In other words, the free wave part is never zero in the soliton resolution of type II blow-up solutions, regardless of the bubble number. 
\end{abstract}

\section{Introduction} 

\subsection{Background}

In this work we consider the focusing, energy-critical wave equation in 3-dimensional space 
\[
 \left\{\begin{array}{ll} \partial_t^2 u - \Delta u = |u|^4 u, & (x,t) \in \Rm^3 \times \Rm;  \\ (u,u_t)|_{t=0} = (u_0,u_1)\in \dot{H}^1\times L^2. & \end{array} \right. \qquad \hbox{(CP1)}
\]
We focus on the solutions with radial symmetry. It is well known that the solutions to (CP1) satisfy the energy conservation law: 
\[
 E = \int_{\Rm^3} \left(\frac{1}{2}|\nabla u(x,t)|^2 + \frac{1}{2}|u_t(x,t)|^2 - \frac{1}{6}|u(x,t)|^6\right) {\rm d} x = \hbox{Const}. 
\]
The most important invariance among radial solutions is the dilation. Namely, if $u$ is a solution to (CP1), then 
\[ 
 u_\lambda = \frac{1}{\lambda^{1/2}} u\left(\frac{x}{\lambda}, \frac{t}{\lambda}\right), \qquad \lambda\in \Rm^+
\]
is also a solution to (CP1). We call this equation ``energy critical'' because the initial data of $u$ and $u_\lambda$ share the same $\dot{H}^1\times L^2$ norm, i.e. the energy norm for wave equations. 

The global behaviour of solutions in the focusing case is much complicated than that of the defocusing case $\partial_t^2 u - \Delta u = - |u|^4 u$, which was intensively studied in the last few decades of 20th century. In short, all finite-energy solutions in the defocusing case are defined for all time and scatter in both two time directions. Please see \cite{mg1, enscatter1, enscatter2, ss1, ss2, struwe}, for instance. If the energy norm of initial data is small, then the corresponding solution to (CP1) still scatters. However, the situation is different for large solutions. 

\paragraph{Ground states} A typical example of non-scattering solutions to (CP1) is the Talenti-Aubin solution
\[
 W(x) = \left(\frac{1}{3} + |x|^2\right)^{-1/2}. 
\]
This solution comes with the least energy among all (not necessarily radial) stationary solutions to (CP1), i.e. solutions to the elliptic equation $-\Delta u = F(u)$. In fact, all radial finite-energy stationary solutions are exactly $\{0, \pm W_\lambda\}$ with 
\begin{align*}
  W_\lambda = \frac{1}{\lambda^{1/2}} W\left(\frac{x}{\lambda}\right), \qquad \lambda > 0. 
\end{align*}
All these solutions (except for $0$) come with the same energy, thus they are all called ground states. 

\paragraph{Finite time blow-up} Unlike the defocusing case, solutions to (CP1) may blow up in finite time $T_+ \in \Rm^+$. A classical local theory then gives 
\[
 \|u\|_{L^5 L^{10} ([0,T_+)\times \Rm^3)} = + \infty. 
\]
We may further divide finite time blow-up solutions into two types: type I blow-up solution satisfies 
 \[
   \limsup_{t\rightarrow T_+} \|(u,u_t)\|_{\dot{H}^1\times L^2} = + \infty. 
  \]
 This type of solution can be constructed explicitly. Clearly the following solution to (CP1) 
 \[
   u(x,t) = \left(\frac{3}{4}\right)^{1/4} (T_+ -t)^{-1/2}, 
  \]
 blows up as $t\rightarrow T_+$. A smooth cut-off technique and the finite speed of propagation then gives a finite-energy type I blow-up solution. 

\paragraph{Type II blow-up solution} In the contrast, a type II blow-up solution satisfies 
 \[
  \limsup_{t\rightarrow T_+} \|(u,u_t)\|_{\dot{H}^1\times L^2} < +\infty. 
 \]
 Soliton resolution conjecture predicts a type II blow-up solution always decomposes into a superposition of decoupled solitary waves, a free wave (the radiation part) and a small error term as the time tends to the blow-up time. In the radial case, all solitary waves are ground states with small but very different sizes. More precisely we have 
 \[
  \vec{u}(t) = \sum_{j=1}^N \zeta_j (W_{\lambda_j(t)},0) + \vec{u}_L(t) + o(1), \qquad t\rightarrow T_+. 
 \]
 Here $\vec{u} = (u,u_t)$, $u_L$ is a free wave, $\zeta_j \in \{+1,-1\}$; the scale functions $\lambda_j(t) > 0$ satisfy
 \begin{align*}
  &\lim_{t\rightarrow T_+} \frac{\lambda_{j+1}(t)}{\lambda_j(t)} = 0;& & \lim_{t\rightarrow T_+} \frac{\lambda_{1}(t)}{T_+ - t} = 0. 
 \end{align*}
 A solution as above is usually called an $N$-bubble solution. In particular, if $u_L=0$, then we call it a pure $N$-bubble solution. Please note that this soliton resolution conjecture has been verified in the radial case, first by Duyckaerts-Kenig-Merle \cite{se} in the 3-dimensional case via a combination of profile decomposition and channel of energy method, then by Duyckaerts-Kenig-Merle \cite{oddhigh}, Duyckaerts-Kenig-Martel-Merle \cite{soliton4d}, Collot-Duyckaerts-Kenig-Merle \cite{soliton6d} and Jendrej-Lawrie \cite{anothersoliton} in higher dimensions, although the non-radial case is still an open problem(please see Duyckaerts-Jia-Kenig \cite{djknonradial} for a partial result). Recently the author \cite{dynamics3d} gives another proof of this conjecture in 3D radial case and discusses some further quantitative properties of the soliton resolution. 

\paragraph{Existence of type II blow-up solutions} The existence of radial type II blow-up solutions was verified before the first proof of the soliton resolution conjecture. One-bubble blow-up solutions were constructed first in Krieger-Schlag-Tataru \cite{slowblowup1}, then in Krieger-Schlag \cite{slowblowup2} and Donninger-Huang-Krieger-Schlag \cite{moreexamples}, with different choices of scale functions $\lambda_1(t)$. Similar type II blow-up solutions in higher dimensions have been discussed in Hillairet-Rapha\"{e}l \cite{4dtypeII} and Jendrej \cite{5dtypeII}. For the evolution of a (weak) solution beyond a type II blow-up time, one may refer to Krieger-Wong \cite{typeIblowup}. 

\subsection{Main topic and result}

A careful review of Krieger, Schlag and Tataru's work \cite{slowblowup1} shows that the radiation part $u_L$ can be arbitrarily small in the energy space but can not be zero in the example of type II blow-up solution they constructed. It is natural to ask whether a type II blow-up solution without any radiation (i.e. $u_L = 0$) exists or not. In other words, we are interested in the existence of pure bubble blow-up solutions. Previous results did give some hints to the answer. For example, Jendrej \cite{nonexistence2bubble} proved that pure two-bubble solution does not exist unless the signs of these two bubbles are the same, in the radial case of all dimensions $d\geq 3$. 

In this work we consider the 3D radial case, and show that the answer to the question above is negative. In other words, pure multi-bubble (or single-bubble) type II blow-up solution does not exist at all, regardless of the bubble number and/or signs of the bubbles. Our main result of this article is 

\begin{theorem} \label{thm main} 
 There does not exist any radial pure bubble type II blow-up solution $u$ to (CP1). In other words, the radiation part $u_L$ in the soliton resolution above can never be zero for any radial finite-time type II blow-up solution. 
\end{theorem}

\begin{remark}
 By the finite speed of propagation, Theorem \ref{thm main} also implies that the free wave $u_L$ can not vanish in any open neighbourhood of the blow-up point $(0,T_+) \in \Rm^3 \times \Rm$. The one-bubble case of this result has been previously known. Please see Duyckaerts-Merle \cite{threshold} and Jendrej \cite{uLlambda}. The novelty of this work lies in its applicability to all bubble numbers/signs. 
\end{remark}
  
\begin{remark}
 The situation of global solutions, i.e. solutions defined for all $t\geq 0$, is much different. Pure one-bubble global solutions do exist, for example, the ground states. In fact,  Duyckaerts-Merle \cite{threshold} classified all solutions to (CP1) with the threshold energy $E = E(W,0)$. It turns out that all radial pure one-bubble global solutions are $\pm W$, $\pm W^+$, $\pm W^-$ and their dilations and/or time translations (with some technical assumptions). Here $W^\pm$ are two special solutions satisfying 
\begin{itemize}
 \item Both $\vec{W}^\pm(t)$ converge to $(W,0)$ in the energy space as $t\rightarrow +\infty$;
 \item $W^-$ scatters in the negative time direction; $W^+$ blows up in finite time in the negative time direction. 
\end{itemize}
The author would like to mention that Jendrej \cite{nonexistence2bubble} also prove the nonexistence of pure 2-bubbles with opposite signs for global solutions. A major difference of the global case from the finite-time blow-up case is the lack of a natural reference point when we evaluate the first bubble size $\lambda_1$.  Please see Remark \ref{reference point} for more details. 
\end{remark}

\begin{remark}
 All previously known examples of soliton resolution are one-bubble solutions in the 3D radial case. Whether a multi-bubble solution exists or not is still an open problem, as far as the author knows. In high dimensional space, however, radial solutions with at least two bubbles have been constructed. Please see, for instance, Jendrej \cite{twobubble6d}. 
\end{remark}

 \begin{remark}
  Finite-time blow-up solution to mass-critical Schr\'{o}dinger equation without dispersion does exist. Please see Bourgain-Wang \cite{Schrzero} for more details. 
 \end{remark}

\begin{remark}
One may also consider the relationship between the blow-up behaviour and the radiation part (also called asymptotic profile in some works) in the soliton resolution. Please see Jendrej \cite{uLlambda}, for example. Similar topic can also be considered for other wave-type equations. For example, Jendrej-Lawrie-Rodriguez \cite{radiationpartWM} discussed the radiation part in the soliton resolution of one-bubble blow-up solutions to the $1$-equivariant wave maps. 
\end{remark} 

\subsection{General idea} 

Now we give the general idea and sketch the proof. Without loss of generality, we may assume that the blow-up time $T_+$ is zero and the bubble number is $n \geq 1$.  The basic idea is to utilize the following identity concerning the virial functional 
\[
 \frac{\rm d}{{\rm d} t} \int_{0}^\infty u_t \left(r u_r + \frac{1}{2} u\right) r^2 {\rm d} r = - \int_0^\infty |u_t|^2 r^2 {\rm d} r
\]
and give a contradiction by an integration from a suitable time $t_\ast < 0$ to $0$. To this purpose, we need to give very precise approximation of the solution $u$ for time sufficiently close to zero. The main tools are a few quantitative soliton resolution estimates given in the author's pervious work \cite{dynamics3d}, as well as a new, refined version of these estimates proved in this work. These estimates are proved by the radiation theory of wave equations. The author hopes that this article demonstrates the robustness of this radiation theory. The proof consists of several steps: 

\paragraph{Extraction of radiation} First of all, if $u$ were a radial pure $n$-bubble type II blow-up solution ($n\geq 1$), then no dispersion assumption would guarantee that $u$ is supported in the ball $\{x: |x|<|t|\}$ for small negative time $t$. A combination of small data theory and finite speed of propagation shows that we may extend the domain of $u$ to the region ($t_0$ is a suitable small negative time) 
\[
 \{(x,t): |x|>t\geq 0\} \cup \left(\Rm^3 \times (t_0,0)\right) \cup \{(x,t): |x|>t_0-t\geq 0\},
\]
with a function (called radiation profile) $G_-\in L^2(t_0,+\infty)$ such that 
\[
 \lim_{t\rightarrow -\infty} \int_{r>|t-t_0|} \left(|r u_r(r,t) - G_-(r+t)|^2 + |r u_t(r,t) - G_-(r+t)|^2\right) {\rm d} r = 0. 
\]
In addition, the inequality $\|G_-\|_{L^2(t',0)} > 0$ holds for any $t'<0$. Otherwise $u$ would become a non-radiative solution outside the light cone $\{(x,t): |x|>|t-t'|\}$, which implies that $u$ must be a ground state or zero and gives a contradiction. Please refer to Section 2 for the theory of exterior solutions, which are solutions defined only outside a suitable light cone, the theory of radiation fields and the conception/basic properties of non-radiative solutions. 

\paragraph{Soliton resolution} Given a time $t'\in (t_0,0)$, we may define a radial free wave $v_{t',L}$ in term of the radiation profile $G_-$, which is asymptotically equivalent to $u(x,t+t')$ outside the main light cone, i.e. 
\[
 \lim_{t\rightarrow \pm \infty} \int_{|x|>|t|} |\nabla_{t,x}(u(x,t+t')-v_{t',L}(x,t))|^2 {\rm d} x = 0. 
\]
Since the Strichartz norm of $v_{t',L}$ is dominated by its energy norm, which is in turn dominated by $\delta(t') \doteq \|G_-\|_{L^2(t',0)} \ll 1$, we obtain the soliton resolution 
\[
 u(x,t'+t) = \sum_{j=1}^n \zeta_j W_{\lambda_j(t')}(x)+ v_{t',L}(x,t) + w(t';x,t), \qquad |x|>|t|; 
\]
with the estimates (see Proposition \ref{main tool})
\begin{align*}
 &\frac{\lambda_{j+1}(t')}{\lambda_j(t')} \lesssim_j \delta(t')^2;& &\|\vec{w}(t',\cdot,0)\|_{\dot{H}^1\times L^2} \lesssim_n \delta(t'); 
\end{align*}
which can be further refined by substituting the right hand side $\delta(t')$ by (and substituting $\delta(t')^2$ by $\tau(t')^2$)
\[
 \tau(t') = \sup_{\lambda\leq \lambda_1(t')} \left\|\chi_0 W_{\lambda}^4 v_{t',L}\right\|_{L^1 L^2(\Rm \times \Rm^3)} + \delta(t')^5 \lesssim_1 \delta(t'). 
\]
Here $\chi_0$ is the characteristic function of $\{(x,t): |x|>|t|\}$. This significantly improves the soliton resolution estimates unless strong concentration of $G_-$ happens near $t'$. By the support of $u$, we may also prove that the first bubble size $\lambda_1(t')$ satisfies 
\[
 \lambda_1(t')\approx \left(\int_{t'}^0 G_-(s) {\rm d} s\right)^2 \lesssim |t'| \delta(t')^2. 
\]

\paragraph{Extraction of good times} For each small time $t' < 0$, we may prove that there exists a nonnegative integer $k$ such that 
\begin{equation} \label{L2 concentration intro}
 \|G_-\|_{L^2(t', t'+2^k \lambda_1(t'))} \gtrsim 2^k \delta(t')^2. 
\end{equation}
Otherwise the inequality $\tau(t') \ll \delta(t')^2$ would hold and the soliton resolution given above would be ``too good'' so that the following almost orthogonality holds
\[
 E(u) \approx n E(W,0) + E(v_{t',L}) > n E(W,0), 
\]
which gives a contradiction. A direct corollary follows that we can always find a sequence of time converging to zero, such that the size of the first bubble is relatively large $\lambda_1 (t') \geq \delta(t')^{2+\kappa}$. Here $\kappa > 0$ is a small constant. Indeed, if $\lambda_1(t')$ is always very small, then a combination of this assumption with the lower bound \eqref{L2 concentration intro} leads to a too strong concentration of $G_-$ at the right hand of any given small time, which makes $\|G_-\|_{L^2(t',0)}$ vanish before $t'$ reaches zero and gives a contradiction. Finally we observe that $\tau(t')$ can also be dominated by the classical maximal function $(\mathbf{M} G_-)(t')$. It follows from this observation that
\begin{itemize}
 \item By adjusting the value of $t'$ slightly, we may find a sequence of time $t_\ast\rightarrow 0^-$ such that 
 \begin{align*}
 &\lambda(t_\ast) \gtrsim \delta(t_\ast)^{2+\kappa}; & & \tau(t_\ast) \lesssim \delta(t_\ast)^{2-\kappa}. 
\end{align*}
\item The soliton resolution is ``quite precise'' for most time $t$ in the interval $(t_\ast,0)$, although it might be imprecise occasionally. 
\end{itemize}

\paragraph{Inserting the virial identity} In the final step we insert 
 \[
 \vec{u}(\cdot,t) = \sum_{j=1}^n \zeta_j \left(W_{\lambda_j(t)},0\right)+ \vec{v}_{t,L}(\cdot,0) + \vec{w}(t;\cdot,0), \qquad t\in [t_\ast,0). 
\]
into the virial identity 
\[
 \int_{0}^\infty u_t(r,t_\ast) \left(r u_r(r,t_\ast) + \frac{1}{2} u(r,t_\ast)\right) r^2 {\rm d} r =  \int_{t_\ast}^0 \int_0^\infty |u_t|^2 r^2 {\rm d} r {\rm d} t,
\]
and finish the proof by a contradiction. Please note that the major terms in the identity above come from either the radiation part $v_{t,L}$, which can be written explicitly in term of the radiation profile $G_-$, or the ground state with the largest size (first bubble), whose size can also be approximated quite precisely by the radiation profile. The term $w(t;\cdot,0)$ is a small error term with 
\[
 \|\vec{w}(t;\cdot,0)\|_{\dot{H}^1\times L^2} \lesssim_n \tau(t),
\]
where the upper bound $\tau(t)$ is quite small for $t=t_\ast$ and most time $t \in [t_\ast,0)$. 

\subsection{Structure of this work} 
This work is organized as follows: In Section 2 we first introduce some notations, basic conceptions and preliminary results, including the exterior solutions, radiation fields and theory, as well as a few soliton resolution estimates in term of the radiation part. We then prove a few refined soliton resolution estimates and give a few applications of these estimates in Section 3. After these preparation work, we devote the remaining three sections to the proof of our main theorem. We give an approximation of the size of the first bubble in Section 4, show the existence of ``good times'', i.e. times with large-size first bubble and good soliton resolution, in Section 5, and finally apply the virial identity to finish the proof in Section 6. 

\section{Preliminary results} 

In this section we make a brief review on previously known related theories and results. We start by introducing a few notations.  

\paragraph{Notations} In this article the notation $A \lesssim B$ means that there exists a constant such that the inequality $A \leq c B$ holds. In addition, we use subscript(s) to indicate that the constant depends on the subscript(s) but nothing else. In particular, the notation $\lesssim_1$ implies that the constant $c$ is an absolute constant. The notations $\gtrsim$ and $\simeq$ can be understood in the same manner. 

\paragraph{Space norms} Sometimes it is necessary to consider the restriction of radial $\dot{H}^1$ functions outside a ball, especially when we consider exterior solutions. We let $\mathcal{H}(R)$ be the space of the restrictions of radial functions $(u_0,u_1) \in \dot{H}^1\times L^2$ to the region $\{x: |x|>R\}$, whose norm is defined by 
\[
 \|(u_0,u_1)\|_{\mathcal{H}(R)} = \left(\int_{|x|>R} \left(|\nabla u_0(x)|^2 + |u_1(x)|^2\right){\rm d} x\right)^{1/2}.
\]
In particular, we define $\mathcal{H} = \mathcal{H}(0)$ to be the Hilbert space of radial $\dot{H}^1\times L^2$ functions. 

\paragraph{Nonlinearity} We use the notation $F(u) = |u|^4 u$ throughout this work, unless specified otherwise. 

\subsection{Exterior solutions}

In order to avoid the irrelevant blow up of solutions near the origin, which is believed to be unpredictable in some sense, it is helpful to introduce solutions to (CP1) defined outside a suitable light cone. We first introduce a few notations. Given $R\geq 0$, we define 
\[
 \Omega_R = \{(x,t)\in \Rm^3 \times \Rm: |x|>|t|+R\}
\]
and use the notation $\chi_R$ for the characteristic function of $\Omega_R$. Similarly we let 
\[
 \Omega_{R_1,R_2} = \{(x,t)\in \Rm^3 \times \Rm: |t|+R_1<|x|<|t|+R_2\}
\]
be the channel-like region and let $\chi_{R_1,R_2}$ be the corresponding characteristic function. We also define $Y$ norm to be the regular $L^5 L^{10}$ Strichartz norm. More precisely, given a time interval $J$, we define
\[
 \|u\|_{Y(J)} = \|u\|_{L^5 L^{10}(J \times \Rm^3)} = \left(\int_J \left(\int_{\Rm^3} |u(x,t)|^{10} {\rm d} x\right)^{1/2} {\rm d} t\right)^{1/5}. 
\]
This $Y$ norm will be frequently used in a combination with the characteristic function defined above. For example, we have 
\[
 \|\chi_R u\|_{Y(J)} = \left(\int_J \left(\int_{|x|>|t|+R} |u(x,t)|^{10} {\rm d} x\right)^{1/2} {\rm d} t\right)^{1/5}.
\]

\paragraph{Exterior solutions} Let $u, F$ be functions defined in the region 
\[
 \Omega =  \{(x,t): |x|>|t|+R,\, t\in (-T_1,T_2)\}\subseteq \Omega_R, \qquad T_1,T_2 \in \Rm^+\cup\{+\infty\}. 
\]
We call $u$ an exterior solution to the following linear wave equation 
\[
  \left\{\begin{array}{l} \partial_t^2 u - \Delta u = F(x,t), \qquad (x,t)\in \Omega;\\ (u,u_t)|_{t=0} = (u_0,u_1) \in \mathcal{H}; \end{array}\right.
\]
if and only if the inequalities $\|\chi_R u\|_{Y(J)} < +\infty$ and $\|\chi_R F\|_{L^1 L^2(J \times \Rm^3)} < +\infty$ holds for any bounded closed time interval $J \subset (-T_1,T_2)$ such that 
 \begin{equation} \label{def of exterior sol}
  u = \mathbf{S}_L (u_0,u_1) + \int_0^t \frac{\sin (t-t')\sqrt{-\Delta}}{\sqrt{-\Delta}} [\chi_R(\cdot,t') F(\cdot,t')] {\rm d} t', \qquad (x,t) \in \Omega
 \end{equation}
Here $\mathcal{S}_L(u_0,u_1)$ is the free wave with initial data $(u_0,u_1)$. Please note that the product $\chi_R F$ is always zero outside the region $\Omega_R$. Although the initial data $(u_0,u_1)$ are defined in the whole space above, finite speed of wave propagation implies the values in the ball $\{x: |x|\leq R\}$ are actually irrelevant. In other words, it is sufficient to specify the initial data $(u_0,u_1)\in \mathcal{H}(R)$. We may define an exterior solution $u$ to nonlinear wave equations in the same manner. For example, $u$ is a solution to
 \[
  \left\{\begin{array}{l} \partial_t^2 u - \Delta u = F(u), \qquad (x,t)\in \Omega;\\ (u,u_t)|_{t=0} = (u_0,u_1) \end{array}\right.
 \]
if and only if the inequality $\|\chi_R u\|_{Y(J)} < +\infty$ holds for any bounded closed time interval $J \subset (-T_1,T_2)$, which also implies that $\|\chi_R F(u)\|_{L^1L^2(J\times \Rm^3)} < +\infty$, and the identity \eqref{def of exterior sol} holds with $F(x,t) = F(u(x,t))$.  More details about exterior solutions can be found in Duyckaerts-Kenig-Merle \cite{oddtool}. 
 
\paragraph{Local theory} The local well-posedness and the continuous dependence of exterior solutions on the initial data (perturbation theory) immediately follows from a combination of the Strichartz estimates (see \cite{strichartz} for instance)
\[
 \sup_{t} \|\vec{u}(t)\|_{\mathcal{H}} + \|u\|_{L^5 L^{10}} \lesssim_1 \|\vec{u}(0)\|_{\mathcal{H}} + \|(\partial_t^2 -\Delta) u\|_{L^1 L^2}
\]
and a standard fixed-point argument. This argument is similar to those in the whole space $\Rm^3$ and somewhat standard nowadays. Please refer to  \cite{loc1, ls} for local well-posedness and \cite{kenig, shen2} for perturbation theory, as a few examples.  

\subsection{Radiation fields}

The theory of radiation fields is quite useful in the discussion of wave equations, especially the asymptotic behaviour of solutions. It has a history of more than half a century. Please see, Friedlander \cite{radiation1, radiation2} for instance. The following version of statement comes from Duyckaerts-Kenig-Merle \cite{dkm3}.

\begin{theorem}[Radiation field] \label{radiation}
Assume that $d\geq 3$ and let $u$ be a solution to the free wave equation $\partial_t^2 u - \Delta u = 0$ with initial data $(u_0,u_1) \in \dot{H}^1 \times L^2(\Rm^d)$. Then
\[
 \lim_{t\rightarrow \pm \infty} \int_{\Rm^d} \left(|\nabla u(x,t)|^2 - |u_r(x,t)|^2 + \frac{|u(x,t)|^2}{|x|^2}\right) {\rm d}x = 0
\]
 and there exist two functions $G_\pm \in L^2(\Rm \times \mathbb{S}^{d-1})$ such that
\begin{align*}
 \lim_{t\rightarrow \pm\infty} \int_0^\infty \int_{\mathbb{S}^{d-1}} \left|r^{\frac{d-1}{2}} \partial_t u(r\theta, t) - G_\pm (r\mp t, \theta)\right|^2 {\rm d}\theta {\rm d}r &= 0;\\
 \lim_{t\rightarrow \pm\infty} \int_0^\infty \int_{\mathbb{S}^{d-1}} \left|r^{\frac{d-1}{2}} \partial_r u(r\theta, t) \pm G_\pm (r\mp t, \theta)\right|^2 {\rm d}\theta {\rm d} r & = 0.
\end{align*}
In addition, the maps $(u_0,u_1) \rightarrow \sqrt{2} G_\pm$ are bijective isometries from $\dot{H}^1 \times L^2(\Rm^d)$ to $L^2 (\Rm \times \mathbb{S}^{d-1})$. 
\end{theorem}

In this article the functions $G_\pm$ are called the radiation profiles of the free wave $u$, or equivalently, of the corresponding initial data $(u_0,u_1)$. This is clear that $u$ is radial if and only if the radiation profile is independent of the angle $\theta$. We may also give explicit formula of free waves, as well as the radiation profile in the other time direction, in term of the radiation profile in one of the time directions. Please see \cite{newradiation, shenradiation} for example. In this work we only need to utilize a simple case, i.e. the radial case in the 3-dimensional space. Indeed we have 
\begin{align}
 & u(r,t) = \frac{1}{r} \int_{t-r}^{t+r} G_-(s) {\rm d}s; & & G_+(s) = - G_-(-s). \label{basic identity radiation free wave}
\end{align}
This implies that we may uniquely determine a radial free wave by specifying the values of both its radiation profiles $G_\pm \in L^2(\Rm^+)$ for $s\in \Rm^+$. A direct calculation shows that the initial data $(u_0,u_1)$ can be given by 
\begin{align} \label{initial data by radiation profile} 
 &u_0(r) =  \frac{1}{r} \int_{-r}^{r} G_-(s) {\rm d}s; & & u_1(r) = \frac{G_-(r)-G_-(-r)}{r}.
\end{align}
An integration by parts yields that 
\begin{equation} \label{radiation residue identity}
  \|(u_0,u_1)\|_{\mathcal{H}(R)}^2 = 8\pi \|G_-\|_{L^2(\{s: |s|>R\})}^2 + 4\pi R |u_0(R)|^2. 
\end{equation} 
We may also consider radiation fields and profiles for suitable solutions to inhomogeneous/nonlinear wave equations. 

\begin{lemma} [Radiation fields of inhomogeneous equation] \label{scatter profile of nonlinear solution}
 Assume that $R\geq 0$. Let $u$ be a radial exterior solution to the wave equation
 \[
  \left\{\begin{array}{ll} \partial_t^2 u - \Delta u = F(t,x); & (x,t)\in \Omega_R; \\
  (u,u_t)|_{t=0} = (u_0,u_1) \in \dot{H}^1\times L^2. & \end{array} \right.
 \]
 If $F$ satisfies $\|\chi_R F\|_{L^1 L^2(\Rm\times \Rm^3)}< +\infty$, then there exist unique radiation profiles $G_\pm \in L^2([R,+\infty))$ such that 
 \begin{align}
  \lim_{t\rightarrow +\infty} \int_{R+t}^\infty \left(\left|G_+(r-t) - r u_t (r, t)\right|^2 + \left|G_+(r-t) + r u_r (r, t)\right|^2\right) {\rm d}r & = 0; \label{positive ra}\\
  \lim_{t\rightarrow -\infty} \int_{R-t}^\infty \left(\left|G_-(r+t) - r u_t(r,t)\right|^2 +  \left|G_-(r+t) - r u_r(r,t)\right|^2\right) {\rm d} r & = 0. \label{negative ra}
 \end{align}
 In addition, the following estimates hold for $G_\pm$ given above and the corresponding radiation profiles $G_{0,\pm}$ of the initial data $(u_0,u_1)$:
 \begin{align*}
 4\sqrt{\pi} \|G_- - G_{0,-}\|_{L^2([R,R'])} & \leq \|\chi_{R,R'} F\|_{L^1 L^2((-\infty,0]\times \Rm^3)}, & & R'>R; \\
  4\sqrt{\pi}  \|G_+ - G_{0,+}\|_{L^2([R,R'])} & \leq \|\chi_{R,R'} F\|_{L^1 L^2([0,+\infty)\times \Rm^3)}, & & R'>R.
 \end{align*}
\end{lemma}
The proof can be found in Section 2 (Lemma 2.5 and Remark 2.6) of the author's previous work \cite{dynamics3d}. The author would like to mention that we can actually give an explicit formula 
 \[
  G_+ (s) - G_{0,+} (s) = \frac{1}{2} \int_{0}^\infty (s+t) F(s+t,t) {\rm d} t. 
 \]
If $u$ is an exterior solution to (CP1) defined in $\Omega_R$ with $\|\chi_R u\|_{Y(\Rm)} < +\infty$, then Lemma \ref{scatter profile of nonlinear solution} also applies to this solution because our assumption guarantees $\|\chi_R F(u)\|_{L^1 L^2(\Rm \times \Rm^3)} < +\infty$. The corresponding radiation profiles $G_\pm \in L^2([R,+\infty))$ given in Lemma \ref{scatter profile of nonlinear solution} are called the (nonlinear) radiation profile of $u$. 

\subsection{Asymptotically equivalent solutions}

Assume that $u, v \in \mathcal{C}(\Rm; \dot{H}^1\times L^2)$ and $R\geq 0$. We say that $u$ and $v$ are $R$-weakly asymptotically equivalent if 
\[
 \lim_{t\rightarrow \pm \infty} \int_{|x|>R+|t|} |\nabla_{t,x} (u-v)|^2 {\rm d} x = 0.
\]
In particular, if $R=0$, then we say that they are asymptotically equivalent to each other. Because the integral above only involves the values of $u, v$ in the exterior region $\Omega_R$, the definition above applies to exterior solutions $u$ and $v$ as well.  

\paragraph{Radiation part} Given a radial exterior solution to (CP1) defined in $\Omega_0$, if $u$ is asymptotically equivalent to a free wave $v_L$, then we call $v_L$ the radiation part of $u$ (outside the light cone). In fact, an exterior solution $u$ is asymptotically equivalent to some free wave in $\Omega_0$ if and only if $\|\chi_0 u\|_{Y(\Rm)} < +\infty$. If $\|\chi_0 u\|_{Y(\Rm)}<+\infty$, then we may determine its (nonlinear) radiation profile $G_\pm \in L^2(\Rm^+)$ by Lemma \ref{scatter profile of nonlinear solution}, and then construct a free wave with the same radiation profiles for $s>0$, which is the desired asymptotically equivalent free wave. To see why the condition $\|\chi_0 u\|_{Y(\Rm)} < +\infty$ is also necessary, please refer to \cite{ecarbitrary}. Please note that this conception of radiation part is different from the radiation part $u_L$ in a soliton resolution at the blow-up time $T_+$, as described in the introduction section of this article. 

\paragraph{Non-radiative solutions} A solution $u$ to the free wave equation, or the nonlinear wave equation (CP1), or any other related wave equation is called ($R$-weakly) non-radiative if and only if it is ($R$-weakly) asymptotically equivalent to zero. Non-radiative solution is an important topic in the channel of energy method (see \cite{channeleven, tkm1, channel} for example), which plays an important role in the study of nonlinear wave equations in recent years. We consider two important examples. 

The first example is radial $R$-weakly non-radiative solution of the free wave equation. By the basic theory of radiation field, it is equivalent to saying that the corresponding radiation profiles satisfy $G_\pm (s) = 0$ for $s>R$, or $G_-(s)$ is supported in the interval $[-R,R]$. As a result, we may give the values of weakly non-radiative free wave $u$ in the region $\Omega_R$ explicitly in term of the radiation profile, thanks to \eqref{basic identity radiation free wave}
\[
 u(r,t) = \frac{1}{r} \int_{-R}^R G_-(s) {\rm d} s, \qquad r>|t|+R. 
\]
These non-radiative solutions form a one-dimensional linear space spanned by $1/r$. 
The second example is non-radiative solution to (CP1). In fact, all non-trivial radial non-radiative solutions to (CP1) are exactly the ground states $\pm W_\lambda(x)$, as mentioned in the introduction. When $r \gg \lambda$ is very large, the re-scaled ground state $W_\lambda$ satisfies 
\[
 W_\lambda(x) \approx \frac{\lambda^{1/2}}{r}.
\]
Namely, the asymptotic behaviour of $\pm W_\lambda$ is always similar to a suitable radial (weakly) non-radiative free wave. Indeed this asymptotic equivalence happens for any radial non-radiative solution to an energy critical wave equation, as shown in \cite{ecarbitrary}. 

\begin{remark}
 The standard ground state $W(x)$ in this article is a dilation of (thus slightly different from) the one $(1+|x|^2/3)^{-1/2}$ used in most related works. This choice helps us eliminate unnecessary constant in the calculation regarding radiation profiles. 
\end{remark}

\subsection{Soliton resolution of almost non-radiative solutions} 
The following proposition separates each bubble one-by-one as long as the radiation in the main light cone is sufficiently weak in the sense of Strichartz norms. This proposition proved in the author's previous work \cite{dynamics3d}, as well as the refined version proved in the next section, plays an essential role in the argument of this work. 
\begin{proposition} \label{main tool} 
 Let $n$ be a positive integer. Then there exists a small constant $\delta_0 = \delta_0(n)>0$ and an absolute constant $c_2 \gg 1$, such that if a radial exterior solution $u$ to (CP1) defined in $\Omega_0$ is asymptotically equivalent to a finite-energy free wave $v_L$ with $\delta \doteq \|\chi_0 v_L\|_{Y(\Rm)} < \delta_0$, then one of the following holds: 
 \begin{itemize}
  \item [(a)] There exists a sequence $(\zeta_j, \lambda_j)\in \{+1,-1\}\times \Rm^+$ for $j=1,2,\cdots,J$ with $0 \leq J\leq n-1$ such that  
  \begin{align*}
    \frac{\lambda_{j+1}}{\lambda_j}  \lesssim_j \delta^2, \qquad & j=1,2,\cdots, J-1; \\
   \left\|\vec{u}(\cdot, 0)-\sum_{j=1}^J \zeta_j (W_{\lambda_j},0) - \vec{v}_L(\cdot,0)\right\|_{\dot{H}^1\times L^2} & + \left\|\chi_0 \left(u - \sum_{j=1}^J \zeta_j W_{\lambda_j} \right)\right\|_{Y(\Rm)}  \lesssim_J \delta.
  \end{align*}
  \item[(b)] There exists a sequence $(\zeta_j, \lambda_j)\in \{+1,-1\}\times \Rm^+$ for $j=1,2,\cdots,n$ satisfying 
  \[
     \frac{|\alpha_{j+1}|}{|\alpha_j|} \lesssim_j \delta, \qquad j=1,2,\cdots, n-1;
  \]
  such that $u$ satisfies the following soliton resolution estimate in the exterior region  
  \begin{align*}
    \left\|\vec{u}(\cdot,0)-\sum_{j=1}^n (W^{\alpha_j},0) - \vec{v}_L(\cdot,0)\right\|_{\mathcal{H}(c_2 \lambda_n)} + \left\|\chi_{c_2 \lambda_n} \left(u - \sum_{j=1}^n W^{\alpha_j} \right)\right\|_{Y(\Rm)} & \lesssim_n \delta.
  \end{align*}
 \end{itemize}
 \end{proposition} 

\begin{remark} 
 The proposition here is slightly different from the original one given in \cite{dynamics3d} in two aspects 
 \begin{itemize}
  \item Proposition 4.1 in \cite{dynamics3d} applies to radial solutions of (CP1) defined in smaller exterior region $\Omega_{R}$, as long as it is still (weakly) asymptotically equivalent to a free wave with a small Strichartz norm. For the sake of simplicity we only use a weaker version here. Please see Remark 4.2 of \cite{dynamics3d} for more details. 
  \item The original proposition utilize a single parameter $\alpha$ to represent a ground state 
  \[
   W^\alpha = \frac{1}{\alpha} \left(\frac{1}{3} + \frac{|x|^2}{\alpha^4}\right)^{-1/2}, \qquad \alpha \in \Rm\setminus \{0\}. 
  \]
  We substitute this single parameter by two parameters $\zeta$ and $\lambda$. It is not difficult to see 
  \[
   W^\alpha = \zeta W_\lambda \quad \Longleftrightarrow \quad \alpha = \zeta \lambda^{1/2}.  
  \]
 \end{itemize}
\end{remark} 

\begin{remark} \label{Jbubble remark}
 We call a radial exterior solution $u$ defined in $\Omega_0$ a $J$-bubble solution if and only if $u$ is asymptotically equivalent to a free wave $v_L$ with $\|\chi_0 v_L\|_{Y(\Rm)} < \delta_0(J+1)$ and the soliton resolution given by Proposition \ref{main tool} comes with exactly $J$ bubbles. 
\end{remark}

The two lemmata below help us compare two asymptotically equivalent solutions. They are necessary in the proof of Proposition \ref{main tool}, as well as the refined soliton resolution estimates given in this work. Their proof depends on the theory of radiation fields. Please refer to Section 4 of \cite{dynamics3d}. 

\begin{lemma} \label{lemma connection} 
 There exist absolute positive constants $\varepsilon_1$, $\beta$, $\eta$ such that if $0\leq R_1<R_2$ and
 \begin{itemize} 
  \item $u$ is an exterior solution to (CP1) and $S$ is an exterior solution to the equation 
 \[
  (\partial_t^2 -\Delta) S= F(S) + e(x,t),
 \] 
 both in the region $\Omega_{R_1}$, with $\|\chi_{R_1} u\|_{Y(\Rm)}, \|\chi_{R_1} S\|_{Y(\Rm)}, \|\chi_{R_1} e(x,t)\|_{L^1 L^2(\Rm\times \Rm^3)} < +\infty$; 
 \item both $u$, $S$ are asymptotically equivalent to each other in $\Omega_{R_1}$; 
 \item $u$, $S$ and $w=u-S$ satisfy the following inequalities
 \begin{align*}
  \varepsilon \doteq \|\vec{w}(\cdot,0)\|_{\mathcal{H}(R_2)}  + \|\chi_{R_1,R_2} e(x,t)\|_{L^1 L^2(\Rm \times \Rm^3)} \qquad & \\
   + \|\chi_{R_2}\left(F(u) - F(S)- e(x,t)\right)\|_{L^1 L^2(\Rm \times \Rm^3)} & \leq \varepsilon_1; \\
   \|\chi_{R_1,R_2} S\|_{Y(\Rm)} & \leq \eta; \\
  \sup_{R_1\leq r\leq R_2} \left(r^{1/2} |w(r,0)|\right) & \leq \beta;
 \end{align*}
 \end{itemize}
 then we have 
 \begin{align*}
   \left\|\vec{w}(\cdot,0)\right\|_{\mathcal{H}(R_1)} + \|\chi_{R_1} w\|_{Y(\Rm)} & \lesssim_1 R_1^{1/2}|w(R_1,0)| +\varepsilon.
 \end{align*}
\end{lemma}

\begin{lemma} \label{lemma connection 2} 
 Let $\eta$ be the constant in Lemma \ref{lemma connection}. There exists an absolute positive constant $\varepsilon_2$ such that if $3 R_2/4 \leq R_1<R_2$ and
 \begin{itemize} 
  \item $u$ is an exterior solution to (CP1) and $S$ is an exterior solution to the equation 
 \[
  (\partial_t^2 -\Delta) S= F(S) + e(x,t). 
 \] 
 with $\|\chi_{R_1} u\|_{Y(\Rm)}, \|\chi_{R_1} S\|_{Y(\Rm)}, \|\chi_{R_1} e(x,t)\|_{L^1 L^2} < +\infty$. 
 \item Solutions $u$, $S$ are asymptotically equivalent to each other in $\Omega_{R_1}$. 
 \item $u$, $S$ and $w=u-S$ satisfy the following inequalities
 \begin{align*}
  \varepsilon \doteq \|\vec{w}(\cdot,0)\|_{\mathcal{H}(R_2)}  + \|\chi_{R_1,R_2} e(x,t)\|_{L^1 L^2(\Rm \times \Rm^3)} \qquad  & \\
 + \|\chi_{R_2}\left(F(u) - F(S)- e(x,t)\right)\|_{L^1 L^2(\Rm \times \Rm^3)} & \leq \varepsilon_2; \\
   \|\chi_{R_1,R_2} S\|_{Y(\Rm)} & \leq \eta; 
 \end{align*}
 \end{itemize}
 Then we have 
 \begin{align*}
  \|\chi_{R_1} w\|_{Y(\Rm)} + \left\|(w(\cdot,0), w_t(\cdot,0))\right\|_{\mathcal{H}(R_1)} \lesssim_1 \varepsilon. 
 \end{align*}
\end{lemma}

\section{Refinement of soliton resolution}

In this section we give a refinement of Proposition \ref{main tool}. This gives a better soliton resolution estimate unless the radiation part $v_L$ concentrates in a region around the origin $(0,0)$ with a size smaller or comparable to the size of the first bubble. In the subsequent sections we will show that this concentration is a rare situation in general. 

\subsection{Refined soliton resolution estimates}

The main result of this section is 

\begin{proposition} \label{refined main tool} 
 Given a positive integer $N$, let $u$ be a radial $N$-bubble solution to (CP1) with a small radiation part $v_L$. Suppose that Proposition \ref{main tool} gives the following soliton resolution
 \[
  u = \sum_{j=1}^N \zeta_j W_{\lambda_j}(x)+ v_L + w. 
 \]
 We define 
 \begin{align*}
  &\delta = \|\chi_0 v_L\|_{Y(\Rm)};& &\tau = \sup_{\lambda\leq \lambda_1} \left\|\chi_0 W_{\lambda}^4 v_L\right\|_{L^1 L^2} + \delta^5\lesssim_1 \delta. 
 \end{align*}
 If $\delta < \delta(N)$ is sufficiently small, then we must have 
 \begin{align*}
  &\frac{\lambda_{j+1}}{\lambda_j} \lesssim_j \tau^2, \; j=1,2,\cdots,N-1;& &\left\|\vec{w}(\cdot,0)\right\|_{\mathcal{H}} \lesssim_N \tau. 
 \end{align*}
\end{proposition}

\begin{remark}
 The same estimates in Proposition \ref{refined main tool} still holds under the same assumption if we substitute $\tau$ by 
 \[
  \tau = \max_{j=1,2,\cdots,n} \left\|\chi_0 W_{\lambda_j}^4 v_L\right\|_{L^1 L^2(\Rm \times \Rm^3)} + \|\chi_0 v_L\|_{Y(\Rm)}^5. 
 \]
 The proof is completely identical. 
\end{remark}

\begin{proof}[Proof of Proposition \ref{refined main tool}]
 We need to make a careful review of the proof for Proposition \ref{main tool} given in \cite{dynamics3d}. We first recall the following notations 
  \begin{align*}
  &S_n = \sum_{j=1}^n \zeta_j W_{\lambda_j} (x) + v_L, & & w_n = u - S_n. 
 \end{align*}
  The approximated solution $S_n$ satisfies the approximated equation 
  \begin{align*}
   &(\partial_t^2 - \Delta) S_n = F(S_n) + e_n;& & e_n = \sum_{j=1}^n \zeta_j F(W_{\lambda_j}) - F(S_n).
  \end{align*}
 The idea is to compare $u$ with each approximated solution $S_n$, starting from $S_0 = v_L$. If $u$ deviates from $S_n$ when $r$ decreases, we add another bubble $\zeta_{n+1} W_{\lambda_{n+1}}$ for compensation and show that $u$ is close to $S_{n+1} = S_n + \zeta_{n+1} W_{\lambda_{n+1}}$. The proof consists of the following steps:
 \begin{itemize}
  \item[(I)] In the first step we prove Proposition \ref{main tool} with $n=1$. We compare $u$ with $v_L$. It is clear that $\|\chi_0 e_0\|_{L^1 L^2} = \|\chi_0 v_L\|_{Y(\Rm)} = \delta^5$. If $w_0 = u-v_L$ satisfies 
  \[
   \sup \left(r^{1/2} |w_0(r,0)|\right) \leq \beta,
  \]
  where $\beta$ (and $\eta$ below) is the constant in Lemma \ref{lemma connection}, then we may apply Lemma \ref{lemma connection} to verify that case (a) of the conclusion with $J=0$ holds. On the other hand, if 
  \[
   \sup \left(r^{1/2} |w_0(r,0)|\right) > \beta, 
  \] 
  we may select a sufficiently large constant $c_2$ and $\beta_1 \doteq c_2^{1/2} (1/3 + c_2^2)^{-1/2} < \beta/2$, let 
  \begin{align}
   &R_1= \max \left\{r: r^{1/2} |w_0(r,0)| = \beta_1\right\}; & & \zeta_1 = \hbox{sign}( w_0(R_1,0)); & & \lambda_1 = c_2^{-1} R_1; \label{def of lambda 1} 
  \end{align}
  and apply Lemma \ref{lemma connection} on $u$ and the approximated solution $S_1 = \zeta_1 W_{\lambda_1} +v_L$ to deduce that case (b) of the conclusion holds for $n=1$. 
  \item[(II)] In the second step we assume that Proposition \ref{main tool} holds for a positive integer $n$ and show that it also holds for $n+1$. We first choose a small constant $c_1 = c_1(n)$ such that $\|\chi_{0,c} W\|_{Y(\Rm)} < \eta/3n$ is sufficiently small. It suffices to consider solutions satisfying (b) in the conclusion for $n$. An application of Lemma \ref{lemma connection 2} immediately gives
  \[
   \|\vec{w}_n(\cdot,0)\|_{\mathcal{H}(c_1\lambda_n)} + \|\chi_{c_1\lambda_n} w_n\|_{Y(\Rm)} + \|\chi_{c_1 \lambda_n}(F(u)-F(S_n)-e_n)\|_{L^1 L^2} \lesssim_n \delta. 
  \] 
  We also have $\|\chi_0 e_n\|_{L^1 L^2} \lesssim_n \delta$. There are still two cases. If 
  \[
   \sup r^{1/2} |w_n(r,0)| \leq \beta,
  \]
  we may apply Lemma \ref{lemma connection} on $u$ and $S_n$ to verify that part (a) in the conclusion holds for $n+1$ with $J=n$. On the other hand, if 
  \[
   \sup r^{1/2} |w_n(r,0)| > \beta,
  \] 
  we may choose
  \begin{align*}
   &R_1= \max \left\{r: r^{1/2} |w_n(r,0)| = \beta_1\right\}; &
    & \zeta_{n+1} = \hbox{sign}( w_n(R_1,0)); & & \lambda_{n+1} = c_2^{-1} R_1;
  \end{align*}
  and apply Lemma \ref{lemma connection} on $u$ and the approximated solution $S_{n+1} = S_n + \zeta_{n+1} W_{\lambda_{n+1}}$ between the radii $R_1$ and $R_2 = c \lambda_n$ to deduce that case (b) of the conclusion holds for $n+1$. 
 \end{itemize}
 Please note that the following scale separation estimate plays an essential role in this argument 
 \[
 \left\|\chi_0 (W_{\lambda})^k (W_{\lambda'})^{5-k}\right\|_{L^1 L^2(\Rm\times \Rm^3)} \lesssim_1 \min\left\{\left(\frac{\lambda}{\lambda'}\right)^{1/2}, \left(\frac{\lambda'}{\lambda}\right)^{1/2}\right\}, \qquad k=1,2,3,4. 
 \]
 The proof of Proposition \ref{refined main tool} follows the same procedure but relies on finer upper bound estimates. We divide the proof into several steps ($1\leq n\leq J-1$): 
 \begin{align*}
  &\|\vec{w_1}(\cdot,0)\|_{\mathcal{H}(c_2 \lambda_1)} + \|\chi_{c_2 \lambda_1} w_1\|_{Y(\Rm)} \lesssim_1 \tau; & &(P_0)\\
  &\left\{\begin{array}{l} \lambda_{n+1}/\lambda_{n} \lesssim_n \tau^2; \\ \|\vec{w}_{n+1} (\cdot,0)\|_{\mathcal{H}(c_2 \lambda_{n+1})} + \|\chi_{c_2 \lambda_{n+1}} w_{n+1}\|_{Y(\Rm)} \lesssim_n \tau; \\ \end{array}\right. & & (P_n)\\
  &\|\vec{w}_J (\cdot,0)\|_{\mathcal{H}} + \|\chi_{0} w_J\|_{Y(\Rm)} \lesssim_J \tau.& &(P_J)
  \end{align*}
 First of all, we apply Lemma \ref{lemma connection} on $u$ and the approximated solution $S_1$ with $R_1 = c_2 \lambda_1$ and $R_2\rightarrow +\infty$ to deduce that 
 \begin{align*}
  \left\|\vec{w}_1(\cdot,0)\right\|_{\mathcal{H}(c_2 \lambda_1)} + \|\chi_{c_2 \lambda_1}w_1\|_{Y(\Rm)} \lesssim_1 \|\chi_{c_2 \lambda_1} e_1\|_{L^1 L^2} \lesssim_1 \tau. 
 \end{align*}
 Here we use the estimate
 \begin{align*}
  \|\chi_0 e_1\|_{L^1 L^2} \lesssim_1 \|\chi_0 W_{\lambda_1}^4 v_L\|_{L^1 L^2} + \|\chi_0 v_L^5\|_{L^1 L^2} \lesssim_1 \tau. 
 \end{align*}
 This verifies $(P_0)$. Next we assume $(P_0), (P_1), \cdots, (P_{n-1})$ and prove $(P_n)$. By the scale separation and the definition of $\tau$ we have 
 \begin{align*}
   \|\chi_0 e_n\|_{L^1 L^2} &\lesssim_n \sum_{1\leq j<k\leq n} \left(\|\chi_0 W_{\lambda_{j}}^4 W_{\lambda_k}\|_{L^1 L^2} + \|\chi_0 W_{\lambda_{j}} W_{\lambda_k}^4\|_{L^1 L^2}\right) + \sum_{j=1}^n \|\chi_0 W_{\lambda_j}^4 v_L\|_{L^1 L^2} + \delta^5\\
   &\lesssim_n  \sum_{1\leq j<k\leq n} \left(\frac{\lambda_k}{\lambda_j}\right)^{1/2} + \tau\\
   &\lesssim_n \tau. 
 \end{align*} 
 It immediately follows that 
 \begin{align*}
  \|\chi_{c_2 \lambda_n} (F(u)-F(S_n) -e_n)\|_{L^1 L^2} & \lesssim_1 \left(\|\chi_{c_2 \lambda_n} w_n\|_{Y(\Rm)}^4 + \|\chi_0 S_n\|_{Y(\Rm)}^4 \right)\|
  \chi_{c_2 \lambda_n} w_n \|_{Y(\Rm)} \\
  & \qquad + \|\chi_0 e_n\|_{L^1 L^2}\\
  & \lesssim_n \tau. 
 \end{align*}
 Given any small constant $c < c_2$, which will be determined later, when $\delta < \delta(n,c)$ is sufficiently small, we may apply Lemma \ref{lemma connection 2} (for multiple times) to deduce 
 \begin{align*}
  \|w_n(\cdot,0)\|_{\mathcal{H}(c \lambda_n)} + \|\chi_{c\lambda_n} w_n \|_{Y(\Rm)} + \|\chi_{c\lambda_n}(F(u)-F(S_n)-e_n)\|_{L^1 L^2} \lesssim_{n,c} \tau. 
 \end{align*}
 Without loss of generality, we may assume that $c$ is a sufficiently small number such that 
 \[
   \|\chi_{0,c} W\|_{Y(\Rm)} \leq \frac{\eta}{3n}.
 \]
 Now we assume that $C > C(n,c)$ is a large constant to be determined later, where the lower bound $C(n,c)$ guarantees that
 \begin{equation} \label{estimate cC}
  \|\vec{w}_n (\cdot,0)\|_{\mathcal{H}(c\lambda_n)}  + \|\chi_{c\lambda_n} w_n\|_{Y(\Rm)}  + \|\chi_{c\lambda_n} (F(u)-F(S_n)-e_n)\|_{L^1 L^2} \leq C \tau, 
 \end{equation}
 which implies (according to \eqref{radiation residue identity})
 \[
  \sup_{r\geq c \lambda_n} \left(r^{1/2} |w_n (r,0)|\right) < C \tau;
 \]
 and define a sequence $c\lambda_n > r_0 > r_1 > r_2 > \cdots$ by 
 \[
  r_k = \max\{r>0: r^{1/2} |w_n (r,0)| = 2^k C \tau\}, \qquad k=0,1,2, \cdots, \left\lfloor\log_2 \frac{\beta_1}{C \tau}\right\rfloor.
 \]
 When $\delta<\delta(n,c,C)$ is sufficiently small, we may apply Lemma \ref{lemma connection} between radii $r_k$ and $c\lambda_n$ to deduce
 \[
  \|\vec{w}_n(\cdot,0)\|_{\mathcal{H}(r_k)} + \|\chi_{r_k} w_n\|\lesssim_{n} 2^k C\tau, \qquad k\geq 0. 
 \]
 We apply Lemma \ref{scatter profile of nonlinear solution} on $w_n$ and deduce 
 \begin{align*}
  \|G_\pm\|_{L^2(r_{k},r_{k-1})} & \lesssim_1 \left\|\chi_{r_{k}, r_{k-1}} \left(F(u) - F(S_n) - e_n\right)\right\|_{L^1 L^2} \\
  & \lesssim_n \tau + \left\|\chi_{r_{k}, r_{k-1}} \left(F(w_n+S_n) - F(S_n)\right)\right\|_{L^1 L^2}\\
  & \lesssim_n \tau + \left(\|\chi_{r_{k}, r_{k-1}} w_n\|_{Y(\Rm)}^4 + \|\chi_{r_{k}, r_{k-1}} S_n\|_{Y(\Rm)}^4\right) \|\chi_{r_{k}, r_{k-1}} w_n \|_{Y(\Rm)}\\
  & \lesssim_n \tau + \left((2^k C \tau)^4 + \left(\frac{r_{k-1}}{\lambda_n}\right)^{2/5} + \delta^4\right) 2^k C \tau \\
  & \lesssim_n \left(C^{-1} 2^{-k} + (2^k C \tau)^{4/5} + \left(\frac{r_{k-1}}{\lambda_n}\right)^{2/5}\right) 2^k C \tau. 
 \end{align*}
 Here $G_\pm$ is the corresponding radiation profile of $\vec{w}_n (0)$. By the definition of $r_k$, the explicit formula \eqref{initial data by radiation profile}, and Cauchy-Schwarz, we have 
 \begin{align*}
  \left|r_k^{1/2} 2^k C \tau - r_{k-1}^{1/2} 2^{k-1} C\tau\right| & = \left|r_k |w_n (r_k,0)| - r_{k-1} |w_n (r_{k-1},0)|\right|\\
  & \leq  \left|r_k w_n (r_k,0) - r_{k-1} w_n (r_{k-1},0)\right|\\
  & \leq \int_{r_k}^{r_{k-1}} \left(|G_+(s)|+|G_-(s)|\right) {\rm d} s\\
  & \lesssim_n \left(C^{-1} 2^{-k} + (2^k C \tau)^{4/5} + \left(\frac{r_{k-1}}{\lambda_n}\right)^{2/5}\right) 2^k C \tau (r_{k-1}-r_k)^{1/2}. 
 \end{align*}
 Thus 
 \begin{align*}
  \left|\frac{2r_{k}^{1/2}}{r_{k-1}^{1/2}} - 1\right|  \lesssim_n \left(C^{-1} 2^{-k} + (2^k C \tau)^{4/5} + \left(\frac{r_{k-1}}{\lambda_n}\right)^{2/5}\right) \left(1-\frac{r_k}{r_{k-1}}\right)^{1/2}.
 \end{align*}
 Therefore there exists a constant $c_n$ (depending on $n$ only) such that 
 \begin{align} \label{ratio estimate rk}
   \frac{2r_{k}^{1/2}}{r_{k-1}^{1/2}} \leq 1 + c_n\left(C^{-1} 2^{-k} + (2^k C \tau)^{4/5} + \left(\frac{r_{k-1}}{\lambda_n}\right)^{2/5}\right).
 \end{align}
 Now we choose $c=c(n)$ and $C=C(n,c)=C(n)$ be sufficiently small/large constants, as well as an additional large constant $N=N(n) \in \mathbb{Z}^+$ such that 
 \begin{align*}
  &c_n c^{2/5} < \frac{1}{6};& & c_n/C < \frac{1}{6};& &c_n \left(\beta_1 2^{-N}\right)^{4/5} < \frac{1}{6}. 
 \end{align*}
 As a result, if $\delta < \delta(n)$ is sufficiently small, then 
 \[
  \frac{2r_{k}^{1/2}}{r_{k-1}^{1/2}} < \frac{3}{2} \; \Longrightarrow \; \frac{r_{k}}{r_{k-1}} < \frac{9}{16}, \qquad \forall k = 1,2,\cdots,  K\doteq \left\lfloor\log_2 \frac{\beta_1}{C \tau}\right\rfloor - N. 
 \]
 Inserting these in \eqref{ratio estimate rk} yields ($k=1,2,\cdots,K$)
\begin{align*}
  \frac{2r_{k}^{1/2}}{r_{k-1}^{1/2}} & \leq 1 + \left( \frac{1}{6} \cdot 2^{-k} + c_n (2^k C \tau)^{4/5} + \frac{1}{6}\left(\frac{9^{k-1}}{16^{k-1}}\right)^{2/5}\right)\\
 & \leq \exp  \left( \frac{1}{6} \cdot 2^{-k} + c_n (2^k C \tau)^{4/5} + \frac{1}{6}\left(\frac{9^{k-1}}{16^{k-1}}\right)^{2/5}\right).
\end{align*}
It follows that 
\begin{align*}
 \frac{4^K r_K}{r_0} \leq \exp \sum_{k=1}^K  \left( \frac{1}{3} \cdot 2^{-k} + 2 c_n (2^k C \tau)^{4/5} + \frac{1}{3}\left(\frac{9^{k-1}}{16^{k-1}}\right)^{2/5}\right) \lesssim_1 1. 
\end{align*}
Next we recall that 
\[ 
 c_2 \lambda_{n+1} = \max\{r>0: r^{1/2} |w_n(r,0)| = \beta_1\} < r_K. 
\]
 This immediately gives
 \[
   \lambda_{n+1} \lesssim_1 r_K \lesssim_1 4^{-K} r_0 \lesssim_n \tau^2 \lambda_n, 
 \]
 which also implies that 
 \begin{align*}
   \|\chi_0 e_{n+1}\|_{L^1 L^2} \lesssim_{n} \tau.
 \end{align*}
 Combining the inequality $\tau_{n+1} \lesssim_n \tau^2 \lambda_n$ with \eqref{estimate cC}, we obtain 
 \[
  \|\chi_{c\lambda_n} w_{n+1}\|_{Y(\Rm)} + \|\vec{w}_{n+1} (\cdot,0)\|_{\mathcal{H}(c\lambda_n)} + \|\chi_{c\lambda_n} (F(u)-F(S_{n+1})-e_{n+1})\|_{L^1 L^2} \lesssim_n \tau.
 \]
 We then apply Lemma \ref{lemma connection} on $S_{n+1}$ and $u$ between radii $R_1 = c_2 \lambda_{n+1}$ and $R_2 = c \lambda_n$  to deduce 
 \[
  \|\chi_{c_2 \lambda_{n+1}} w_{n+1}\|_{Y(\Rm)} + \|\vec{w}_{n+1} (\cdot,0)\|_{\mathcal{H}(c_2 \lambda_{n+1})} + \|\chi_{c_2 \lambda_{n+1}} (F(u)-F(w_{n+1})-e_{n+1})\|_{L^1 L^2} \lesssim_{n} \tau,
 \]
 as long as $\delta < \delta(n)$ is sufficiently small. This finally verifies $(P_{n})$. In order to complete the proof, it suffices to assume $(P_0), (P_1), \cdots, (P_{J-1})$ and verify $(P_J)$. Our induction hypothesis implies that 
 \begin{align*}
  \|\chi_0 e_J\|_{L^1 L^2} & \lesssim_J \tau; \\
  \|\chi_{c_2 \lambda_{J}} w_{J}\|_{Y(\Rm)} + \|\vec{w}_{J} (\cdot,0)\|_{\mathcal{H}(c_2 \lambda_{J})} + \|\chi_{c_2 \lambda_{J}} (F(u)-F(w_{J})-e_{J})\|_{L^1 L^2} & \lesssim_{J} \tau. 
 \end{align*}
 We may apply Lemma \ref{lemma connection 2} for multiple times to deduce 
 \[
  \|\chi_{c_1 \lambda_{J}} w_{J}\|_{Y(\Rm)} + \|\vec{w}_{J} (\cdot,0)\|_{\mathcal{H}(c_1 \lambda_{J})} + \|\chi_{c_1 \lambda_{J}} (F(u)-F(w_{J})-e_{J})\|_{L^1 L^2}  \lesssim_{J} \tau
 \]
 Here $c_1 = c_1(J)$ is a small constant such that 
 \[
  \|\chi_{0,c_1} W\|_{Y(\Rm)} < \frac{\eta}{3n}. 
 \]
 Our assumption that $u$ is a $J$-bubble solution implies 
 \[
  \sup_{r>0} \left(r^{1/2} |w_J(r,0)|\right) \leq \beta,
 \]
 which enables us to apply Lemma \ref{lemma connection} on $u$ and $S_J$ between radii $0$ and $c_1 \lambda_J$ to obtain 
 \[
  \|\chi_{0} w_{J}\|_{Y(\Rm)} + \|\vec{w}_{J} (\cdot,0)\|_{\mathcal{H}} \lesssim_J \tau.  
 \]
 This completes the proof. 
\end{proof}

Next we give a ways to find an upper bound of $\tau$ by concerning the maximal function of the radiation profile. This upper bound will be used in the subsequent sections. 

\begin{lemma} \label{first lemma upper bound tau}
 Let $v_L$ be a radial free wave with radiation profile $G_-$ and $\lambda$ be a positive number. Then we have 
 \[
  \left\|\chi_0 W_\lambda^4 v_L\right\|_{L^1 L^2} \lesssim_1 \lambda^{1/2} \left(\sup_{r>0} \frac{1}{r}\int_{-r}^r |G_-(s)| {\rm d} s\right).
 \]
\end{lemma}
\begin{proof}
 We split $\Omega_0$ into infinitely many regions $\Omega_0 = \Psi_0 \cup \Psi_1 \cup \Psi_2 \cup \cdots$
 \begin{align*}
  &\Psi_0 = \{(x,t): |t|<|x|<\lambda\};& &\Psi_k = \left\{\max\{|t|,2^{k-1}\lambda\}<|x|<2^k \lambda\right\}, \; k=1,2,\cdots;
 \end{align*}
 and let $\tilde{\chi}_k$ be the characteristic function of $\Psi_k$. We have 
 \begin{align*}
  \left\|\chi_0 W_\lambda^4 v_L\right\|_{L^1 L^2} & \leq \sum_{k=0}^\infty \left\|\tilde{\chi}_k W_{\lambda}^4 v_L\right\|_{L^1 L^2} \\
  & \leq \sum_{k=0}^\infty \left\|\tilde{\chi}_k W_{\lambda}\right\|_{Y(\Rm)}^4 \|\tilde{\chi}_k v_L\|_{Y(\Rm)}\\
  & \lesssim_1 \sum_{k=0}^\infty 2^{-2k} \|\tilde{\chi}_{k} v_L\|_{Y(\Rm)}.
 \end{align*}
 Next we find an upper bound of $\|\tilde{\chi}_{k} v_L\|_{Y(\Rm)}$. For convenience we let 
 \[
  M = \sup_{r>0} \frac{1}{r}\int_{-r}^r |G_-(s)| {\rm d} s. 
 \]
 The explicit formula of $v_L$ in term of the radiation profile $G_-$ yields 
 \begin{align*}
  |v_L(x,t)| = \frac{1}{|x|} \left|\int_{t-|x|}^{t+|x|} G_-(s) {\rm d} s\right| \leq \frac{1}{|x|} \int_{-2|x|}^{2|x|} |G_-(s)| {\rm d} s \leq 2M, \quad |x|>|t|. 
 \end{align*}
 This implies that 
 \[
  \|\tilde{\chi}_k v_L\|_{Y(\Rm)} \lesssim_1 (2^k \lambda)^{1/2} M. 
 \]
 Inserting this into the inequality above finishes the proof.
\end{proof}

\subsection{One application of refined estimates}

In this subsection we show that a $J$-bubble solution to (CP1) typically comes with more energy than $J E(W,0)$ unless the radiation profile $G_-(s)$ of its radiation part concentrates around $s=0$ in some sense. 

\begin{lemma} \label{large energy lemma}
 Given a positive integer $J$, let $u$ be a radial $J$-bubble solution to (CP1) as described in Remark \ref{Jbubble remark}, with radiation part $v_L$ and first bubble size $\lambda_1$. There exists two small constants $c=c (J)$ and $\delta = \delta_1(J)$, such that if the radiation profile $G_-$ of $v_L$ satisfies $\|G_-\|_{L^2(\Rm)} < \delta_1$ and 
  \[
   \|G_-\|_{L^2(-2^k \lambda_1, 2^k \lambda_1)} \leq c\cdot 2^k \|G_-\|_{L^2(\Rm)}^2, \qquad \forall k \geq 0. 
  \]
 Then we must have $E(u) > J E(W,0)$. 
\end{lemma}
\begin{remark}
 The radiation profile $G_-$ of $v_L$ can be uniquely determined by the corresponding (nonlinear) radiation profiles $g_\pm \in L^2(\Rm^+)$ of $u$. In fact we have 
 \[
  G_-(s) = \left\{\begin{array}{ll} g_-(s), & s>0; \\ -g_+(-s), & s<0. \end{array}\right.
 \]
\end{remark}
\begin{proof}[Proof of Lemma \ref{large energy lemma}]
 First of all, we recall the notations $\Psi_k$, $\tilde{\chi}_k$ given in the proof of Lemma \ref{first lemma upper bound tau} and the inequality 
 \[
  \left\|\chi_0 W_\lambda^4 v_L\right\|_{L^1 L^2} \lesssim_1 \sum_{k=0}^\infty 2^{-2k} \|\tilde{\chi}_{k} v_L\|_{Y(\Rm)}. 
 \]
 Here $\lambda \leq \lambda_1$ is a small bubble size. The explicit formula of $v_L$ in terms of $G_-$ reveals that the values of $v_L$ in $\Psi_k$ depends on the values of $G_-$ in the interval $(-2^{k+1}\lambda, 2^{k+1}\lambda)$ only. It follows from the Strichartz estimates that
 \[
  \|\tilde{\chi}_k v_L\|_{Y(\Rm)} \lesssim_1 \|G_-\|_{L^2(-2^{k+1}\lambda, 2^{k+1}\lambda)} \lesssim_1 \|G_-\|_{L^2(-2^{k+1}\lambda_1, 2^{k+1}\lambda_1)} \lesssim_1 c\cdot 2^k \delta^2.
 \] 
 Here for convenience we define $\delta = \|G_-\|_{L^2(\Rm)}$. As a result, we have 
 \[
  \sup_{\lambda\leq \lambda_1} \left\|\chi_0 W_\lambda^4 v_L\right\|_{L^1 L^2} \lesssim_1 c\delta^2. 
 \]
 This implies that if $\delta < \delta(c, J)$ is sufficiently small, then 
 \[
  \tau = \sup_{\lambda\leq \lambda_1} \left\|\chi_0 W_\lambda^4 v_L\right\|_{L^1 L^2} + \|\chi_0 v_L\|_{Y(\Rm)}^5 \lesssim_1 c\delta^2 + \delta^5 \lesssim_1 c \delta^2. 
 \]
 An application of Proposition \ref{refined main tool} then gives the soliton resolution 
 \[
  \vec{u}(\cdot, 0) = \sum_{j=1}^J \zeta_j (W_{\lambda_j}, 0) + \vec{v}_L(\cdot,0) + \vec{w}_J (\cdot,0),
 \]
 with 
 \begin{align*}
  &\frac{\lambda_{j+1}}{\lambda_j} \lesssim_J c^2 \delta^4; & & \|\vec{w}_J(\cdot,0)\|_{\mathcal{H}} \lesssim_J c \delta^2. 
 \end{align*}
 These estimates, as well as the following inequalities by scale separation 
 \begin{align*}
  \int_{\Rm^3} \left|\nabla W_{\lambda_j} \cdot \nabla W_{\lambda_k}\right| {\rm d} x & \lesssim_1 \left(\frac{\lambda_k}{\lambda_j}\right)^{1/2}, & & j<k;\\
  \sum_{\ell = 1}^5 \int_{\Rm^3} W_{\lambda_j}^\ell W_{\lambda_k}^{6-\ell} {\rm d} x & \lesssim_1 \left(\frac{\lambda_k}{\lambda_j}\right)^{1/2}, & & j<k; 
 \end{align*}
 give the following almost orthogonality properties
 \begin{align*}
  \left|\|\vec{u}(\cdot,0)\|_{\mathcal{H}}^2 - J \|W\|_{\dot{H}^1(\Rm^3)}^2 - 8\pi \delta^2\right| & \lesssim_J \sum_{j=1}^J \left| \langle W_{\lambda_j}, v_L(\cdot,0)\rangle_{\dot{H}^1}\right| + c\delta^2;\\
  \left|\int_{\Rm^3} |u(x,0)|^6 {\rm d} x - J \|W\|_{L^6}^6\right| & \lesssim_J \sum_{j=1}^J \int_{\Rm^3} W_{\lambda_j}^5 |v_L(x,0)| {\rm d} x + c\delta^2.  
 \end{align*}
 A combination of them gives the energy estimate 
 \[
  \left|E(u) - J E(W,0) - 4\pi \delta^2\right| \lesssim_J \sum_{j=1}^J \left| \langle W_{\lambda_j}, v_L(\cdot,0)\rangle_{\dot{H}^1}\right| + \sum_{j=1}^J \int_{\Rm^3} W_{\lambda_j}^5 |v_L(x,0)| {\rm d} x + c\delta^2.  
 \]
 Next we give the upper bounds of the right hand side. We start by observing that 
 \begin{align*}
  \langle W_{\lambda_j}, v_L(\cdot,0)\rangle_{\dot{H}^1} & = 4\pi \int_0^\infty \partial_r (r W_{\lambda_j}) \partial_r (r v_L (r,0)) {\rm d} r \\
  & = 4\pi \int_0^\infty \frac{1}{3\lambda_j^{1/2}}\left(\frac{1}{3} + \frac{r^2}{\lambda_j^2}\right)^{-3/2}[G_-(r)+G_-(-r)] {\rm d} r.
 \end{align*}
 This immediately gives 
 \begin{align*}
  \left|\langle W_{\lambda_j}, v_L(\cdot,0)\rangle_{\dot{H}^1} \right| & \lesssim_1 \int_0^{\lambda_j} \!\!\frac{1}{\lambda_j^{1/2}} (|G_-(r)\!+\!|G_-(-r)|) {\rm d} r + \sum_{k=1}^\infty \int_{2^{k-1} \lambda_j}^{2^k \lambda_j} \!\!\frac{\lambda_j^{5/2}}{r^3} (|G_-(r)\!+\!|G_-(-r)|) {\rm d} r\\
  & \lesssim_1 \|G_-\|_{L^2(-\lambda_j, \lambda_j)} + \sum_{k=1}^\infty 2^{-5k/2} \|G_-\|_{L^2(-2^k \lambda_j, 2^k \lambda_j)}\\
  & \lesssim_1 \sum_{k=0}^\infty 2^{-5k/2}\|G_-\|_{L^2(-2^k \lambda_1, 2^k \lambda_1)}\\
  & \lesssim_1 c\delta^2. 
 \end{align*}
 Similarly we have 
 \begin{align*}
  \int_{\Rm^3} W_{\lambda_j}^5 |v_L(x,0)| {\rm d} x & = \int_{|x|<\lambda_j} W_{\lambda_j}^5 |v_L(x,0)| {\rm d} x + \sum_{k=1}^\infty \int_{2^{k-1} \lambda_j<|x|<2^k \lambda_j} W_{\lambda_j}^5 |v_L(x,0)| {\rm d} x \\
  & \lesssim_1 \|v_L(\cdot,0)\|_{L^6(\{x: |x|<\lambda_j\})} + \sum_{k=1}^\infty 2^{-5k/2} \|v_L(\cdot,0)\|_{L^6(\{x: |x|< 2^k \lambda_j\})} \\
  & \lesssim_1 \sum_{k=0}^\infty 2^{-5k/2} \|G_-\|_{L^2(-2^k \lambda_j, 2^k \lambda_j)} \\
  &  \lesssim_1 \sum_{k=0}^\infty 2^{-5k/2} \|G_-\|_{L^2(-2^k \lambda_1, 2^k \lambda_1)} \\
  & \lesssim_1 c \delta^2. 
 \end{align*}
 In summary we obtain 
 \[
  \left|E(u) - J E(W,0) - 4\pi \delta^2\right| \lesssim_J c \delta^2. 
 \]
 This finally verifies that $E(u) > J E(W,0)$, as long as $c = c(J)$ and $\delta < \delta_1(J) \doteq \delta(c,J)$ are both sufficiently small. 
\end{proof}

\section{The first bubble size}

The rest of this work is devoted to the proof of our main theorem. The proof starts with some set-up work. We argue by a contradiction. Without loss of generality, let us assume that $u$ is defined for all small negative times and blows up in the manner of type II at time $t=0$, with a bubble number $n \geq 1$. This bubble number, as well as our assumption on the free wave part $u_L=0$, immediately gives
\[
 E(u) = n E(W,0). 
\]
The assumption $u_L=0$ also implies that 
\begin{equation} \label{support of claim}
 u(x,t) = 0, \qquad |x|>-t>0. 
\end{equation} 
Indeed, the soliton resolution with $u_L = 0$ immediately gives 
\[
 \lim_{t\rightarrow 0^-} \|\vec{u}(\cdot,t)\|_{\mathcal{H}(|t|)} = 0. 
\]
Solving the wave equation backward in time and applying the finite speed of propagation, we immediately obtain \eqref{support of claim}. In addition, recalling the fact $\vec{u}(t) \in \dot{H}^1\times L^2$, we may fix a small negative time $t_0$ and find a small number $r_0 > 0$ such that 
\[
 \|\vec{u}(t_0)\|_{\mathcal{H}(|t_0|-r_0)} \ll 1. 
\]
By small data theory of exterior solutions, we may extend the domain of the solution $u$ if necessary such that $u$ is also defined in the region 
\[
 \{(x,t): |x|>|t|-r_0, t\leq t_0\}
\]
In view of the support of $u$, we may further extend its domain by defining $u(x,t) =0$ for $|x|>t\geq 0$. In summary we may define $u$ in the region (see figure \ref{figure domain})
\[
 \left(\Rm^3 \times [t_0, 0)\right) \cup \{(x,t): |x|>t \geq 0\} \cup \{(x,t): |x|>|t|-r_0, t<t_0\},
\]
such that $u$ is an exterior solution outside each light cone $|x| = |t-t_1|$ for $t_1 \in [-r_0,0)$. In addition, the small data theory of exterior solution and the support of $u$ also implies that the (nonlinear) radiation profiles of $u$ satisfy
\begin{align*} 
 &G_\pm (s) = 0, \quad s>0;& &\|G_-\|_{L^2([-r_0,+\infty)} \ll 1.
\end{align*}

 \begin{figure}[h]
 \centering
 \includegraphics[scale=1.0]{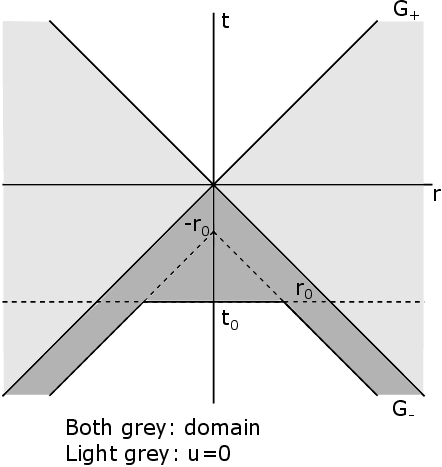}
 \caption{Domain extension of $u$} \label{figure domain}
\end{figure}

For each fixed $t\in [-r_0,0)$, let us consider the time-translated solution $u(\cdot, \cdot + t)$ and restrict its domain to $\Omega_0$. It is not difficult to see that this exterior solution is asymptotically equivalent to the free wave $v_{t,L}$ determined by the radiation profiles
\begin{align*}
 &G_{t,+}(s) = 0, \quad s>0;& &G_{t,-}(s) = G_-(s+t), \quad s>0.
\end{align*}
Clearly we have 
\[
 \|\chi_0 v_{t,L}\|_{Y(\Rm)} \lesssim_1 \|\vec{v}_{t,L}\|_{\mathcal{H}} \lesssim_1 \|G_-\|_{L^2([t,0])} \ll 1.  
\]
We define $\delta(t) = \|G_-\|_{L^2([t,0])}$ for convenience. An application of Proposition \ref{main tool} then implies that there exist $\zeta_j \in \{+1,-1\}$ and scale functions $\lambda_j(t)$ for $j=1,2,\cdots,n$ such that 
\begin{align*}
 &\frac{\lambda_{j+1}(t)}{\lambda_{j}(t)} \lesssim_j \delta^2(t),\quad j=1,2,\cdots, n-1;& &\left\|\vec{u}(t) -\sum_{j=1}^n \zeta_j (W_{\lambda_j(t)},0) - \vec{v}_{t,L}(0)\right\|_{\mathcal{H}} \lesssim_n \delta(t).
\end{align*}
Here we only consider sufficiently small time $t$ and use the assumption on the bubble number. The signs $\zeta_j$ do not depend on the choice of small time $t$ by a continuity argument. For convenience we define
\[
 w(t;x,t') = u(x,t+t') - \sum_{j=1}^n \zeta_j W_{\lambda_j(t)} (x) - v_{t,L}(x,t'), \qquad (x,t') \in \Omega_0. 
\]
We may rewrite the soliton resolution in the form of 
\begin{align*}
 &\vec{u}(\cdot,t) = \sum_{j=1}^n \zeta_j (W_{\lambda_j(t)}, 0) + \vec{v}_{t,L}(\cdot,0) + \vec{w}(t; \cdot,0);& &\|\vec{w}(t;\cdot,0)\|_{\mathcal{H}} \lesssim_n \delta(t). 
\end{align*}
Please note that $\|G_-\|_{L^2(t,0)} > 0$ for any $t<0$. Otherwise $u(\cdot, \cdot+t)$ must be a non-radiative solution, thus either a zero solution or a ground state. This is a contradiction. The following lemma gives a precise estimate on the size of the first bubble. 

\begin{lemma} \label{size of the first bubble lemma}
 There exists a small constant $\delta_0>0$, such that if $u$ is a radial solution of (CP1) defined in the exterior region $\Omega_0$ satisfying the following conditions
 \begin{itemize} 
  \item $\vec{u}(0)$ is supported in the ball $\{x: |x|\leq R\}$; 
  \item the solution is a scattering exterior solution in $\Omega_0$, i.e. $\|\chi_0 u\|_{Y(\Rm)} < +\infty$; 
  \item The (nonlinear) radiation profiles of $u$ satisfy $G_+(s) =0$ for $s>0$ in the positive time direction and $\delta \doteq \|G_-\|_{L^2(\Rm^+)} < \delta_0$ in the negative time direction; 
 \end{itemize}
 then exactly one of the following holds: 
 \begin{itemize}
  \item[(a)] $u$ is a zero-bubble solution. In addition 
  \[
   \left|\int_{0}^R G_-(s) {\rm d} s\right| \lesssim_1 R^{1/2}\delta^5.
  \]
  \item[(b)] $u$ is a solution with at least one bubble. In addition, the sign $\zeta_1$ and size $\lambda_1$ of the first bubble, as given by Proposition \ref{main tool}, must satisfy the inequalities
  \begin{align*}
   &\lambda_1 \lesssim_1 R \delta^2; & &\left|\int_0^R G_-(s) {\rm d} s + \zeta_1 \lambda_1^{1/2}\right| \lesssim_1 R^{1/2} \delta^{9/5}. 
 \end{align*}
 \end{itemize}
\end{lemma}
\begin{proof}
 First of all, the support of $\vec{u}(0)$ implies that $u \equiv 0$ in $\Omega_R$, thus $G_-(s) = 0$ for $s>R$. We use the notation $v_L$ for the radial free wave asymptotically equivalent to $u$ in $\Omega_0$. The radiation profile $g_-$ of $v_L$ is supported in $[0,R]$ with 
 \[
  g_-(s) = G_-(s), \qquad s\in [0,R]. 
 \]
 It follows from the Strichartz estimates that $\|v_L\|_{Y(\Rm)} \lesssim_1 \delta$. As in the first step of proof for Proposition \ref{main tool}, we define $w = u - v_L$ and consider the value of $r^{1/2} w(r,0)$. A direct calculation shows 
 \begin{equation} \label{outer estimate R} 
  \sup_{r\geq R} r^{1/2} |w(r,0)| \lesssim_1 \|\vec{w}(\cdot,0)\|_{\mathcal{H}(R)} = \|\vec{v}_L(\cdot,0)\|_{\mathcal{H}(R)} \lesssim_1 \delta. 
 \end{equation}
 A review of the proof for Proposition \ref{main tool} shows that there are two cases. In the first case the soliton resolution of $\vec{u}(0)$ does not contain any bubble and the following inequality holds
 \[
  \sup_{r > 0} r^{1/2} |w(r,0)| < \beta. 
 \]
 This enable us to apply Lemma \ref{lemma connection} to deduce
 \[
   \|\vec{u}(\cdot,0) - \vec{v}_L(\cdot,0)\|_{\mathcal{H}} \lesssim_1 \|\chi_0 v_L^5\|_{L^1 L^2} \lesssim_1 \delta^5. 
 \]
 This implies that 
 \[
  \|\vec{v}_L(\cdot,0)\|_{\mathcal{H}(R)} \lesssim_1 \delta^5 \quad \Longrightarrow \quad \left|\int_{0}^R G_-(s) {\rm d} s\right| \lesssim_1 R^{1/2}\delta^5. 
 \]
 Here we utilize the explicit formula
 \begin{equation} \label{explicit vL lambda 1}
  v_L(r,0) = \frac{1}{r}\int_0^R G_-(s) {\rm d} s, \qquad r>R. 
 \end{equation}
 In the second case, the following inequality holds and the soliton resolution of $\vec{u}(0)$ comes with at least one bubble
 \[
  \sup_{r>0} r^{1/2} |w(r,0)| \geq \beta. 
 \]
 When $\delta$ is sufficiently small, we choose
 \begin{align*}
  R_1 = \max\left\{r>0: r^{1/2} |w(r,0)| = \delta^{-1/5} W(\delta^{-2/5}) \right\}. 
 \end{align*}
 Since $\delta \ll \delta^{-1/5} W(\delta^{-2/5}) \simeq_1 \delta^{1/5} \ll 1$, we may compare this definition of $R_1$ with \eqref{def of lambda 1} and \eqref{outer estimate R} to deduce that $R_1 \in (c_2 \lambda_1, R)$. Next we apply Lemma \ref{lemma connection} on $u$ and $v_L+ \zeta W_{\lambda}$, with $\lambda = \delta^{2/5} R_1$ and $\zeta = \hbox{sign} (w(R_1,0))$, between the radii $R_1$ and $R_2 \rightarrow +\infty$ to deduce 
 \begin{align}
 \left\|\vec{u}(\cdot, 0) - \vec{v}_L(\cdot, 0) - \left(\zeta W_{\lambda},0\right)\right\|_{\mathcal{H}(R_1)} &\lesssim_1 \left\|\chi_{R_1} \left[F(\zeta W_{\lambda}) - F(v_L + \zeta W_{\lambda})\right]\right\|_{L^1 L^2} \nonumber\\
 & \lesssim_1 \left(\|\chi_{R_1} v_L\|_{Y(\Rm)}^4 + \|\chi_{R_1} W_\lambda\|_{Y(\Rm)}^4\right)\|\chi_{R_1} v_L\|_{Y(\Rm)}\nonumber\\
 & \lesssim_1 \delta^{9/5}. \label{alternative soliton resolution}
 \end{align}
 Combining this with the support of $\vec{u}(0)$ and the explicit expression \eqref{explicit vL lambda 1}, we obtain 
 \begin{align} \label{tail 14kappa}
  &\left\|\frac{\alpha'}{|x|} + \zeta W_{\lambda}\right\|_{\dot{H}^1(\{x: |x|>R\})} \lesssim_1 \delta^{9/5}; & & \alpha' = \int_0^R G_-(s) {\rm d} s. 
 \end{align}
 Since 
 \[
  \left\|\alpha'/|x|\right\|_{\dot{H}^1(\{x: |x|>R\})}\lesssim_1 R^{-1/2} \int_0^R |G_-(s)| {\rm d} s \lesssim_1 \delta, 
 \] 
  the inequality \eqref{tail 14kappa} implies that 
 \[
  \|W_{\lambda}\|_{\dot{H}^1(\{x: |x|>R\})} \lesssim_1 \delta \ll 1 \quad \Longrightarrow \quad \lambda \lesssim_1 R \delta^2. 
 \]
 We observe
 \[
  \partial_r W_\lambda = - \frac{r}{\lambda^{5/2}} \left(\frac{1}{3} + \frac{r^2}{\lambda^2}\right)^{-3/2} = - \frac{\lambda^{1/2}}{r^2} \left(1+\frac{\lambda^2}{3r^2}\right)^{-3/2} = -\frac{\lambda^{1/2}}{r^2} + O_1\left(\frac{\lambda^{5/2}}{r^4}\right), \quad r>R\gg \lambda,
 \]
 which means that 
 \[
  \left\|W_\lambda - \frac{\lambda^{1/2}}{|x|}\right\|_{\dot{H}^1(\{x: |x|>R\})} \lesssim_1 \left\|\frac{\lambda^{5/2}}{r^4}\right\|_{L^2([R,+\infty); r^2 {\rm d} r)} \lesssim_1 \left(\frac{\lambda}{R}\right)^{5/2} \lesssim_1 \delta^5. 
 \]
 Inserting this into \eqref{tail 14kappa}, we obtain 
 \begin{equation} \label{estimate alpha}
  \left\|\frac{\alpha' + \zeta \lambda^{1/2}}{|x|} \right\|_{\dot{H}^1(\{x: |x|>R\})} \lesssim_1 \delta^{9/5}\quad \Longrightarrow \quad |\alpha'+\zeta \lambda^{1/2}| \lesssim_1 R^{1/2}  \delta^{9/5}. 
 \end{equation} 
 Now we figure out the relationship of $(\zeta, \lambda)$ and the sign $\zeta_1$ and size $\lambda_1$ of the first bubble in the original soliton resolution. We recall $c_2 \lambda_1 < R_1$, combine \eqref{alternative soliton resolution} and the original soliton resolution to deduce 
 \[
  \left\|\zeta W_\lambda - \zeta_1 W_{\lambda_1}\right\|_{\dot{H}^1(\{x: |x|>R_1\})} \lesssim_1 \delta \ll 1. 
 \]
 Combining this with the fact that 
 \[
  \|W\|_{\dot{H}^1(\{x: |x|>r\})}\simeq_1 \min\{1, r^{-1/2}\}, 
 \]
 and $\lambda = R_1 \delta^{2/5}$,  we see that $\zeta = \zeta_1$ and $\lambda \simeq_1 \lambda_1$. A similar argument as above shows that 
 \begin{align*}
  \left\|W_\lambda - \frac{\lambda^{1/2}}{|x|}\right\|_{\dot{H}^1(\{x: |x|>R_1\})} & \lesssim_1 \left(\frac{\lambda}{R_1}\right)^{5/2} \lesssim_1 \delta;\\
  \left\|W_{\lambda_1} - \frac{\lambda_1^{1/2}}{|x|}\right\|_{\dot{H}^1(\{x: |x|>R_1\})} & \lesssim_1 \left(\frac{\lambda_1}{R_1}\right)^{5/2} \lesssim_1 \delta. 
 \end{align*}
 A combination of the three estimates above immediately yields 
 \[
  \left\|\frac{\lambda^{1/2}}{|x|} - \frac{\lambda_1^{1/2}}{|x|}\right\|_{\dot{H}^1(\{x: |x|>R_1\})} \lesssim_1 \delta \quad \Longrightarrow \quad \left|\lambda^{1/2}-\lambda_1^{1/2}\right| \lesssim_1 R_1^{1/2} \delta.
 \]
 By the estimate $R_1 \delta^{2/5} = \lambda \lesssim_1 R \delta^2$, we also have $R_1 \lesssim_1 R \delta^{8/5}$. As a result, we obtain the estimate $|\lambda^{1/2}-\lambda_1^{1/2}| \lesssim_1 R^{1/2} \delta^{9/5}$. Combining this with \eqref{estimate alpha}, we finally obtain 
 \[
  \left|\int_0^R G_-(s) {\rm d} s + \zeta_1 \lambda_1^{1/2}\right| \lesssim_1 R^{1/2}\delta^{9/5}. 
 \]
 This finishes the proof. 
\end{proof}

\begin{remark} \label{reference point} 
 We may apply the lemma above on the time-translated version $u(\cdot, \cdot+t)$ of the type II blow-up solution $u$ described at the beginning of this section for small time $t<0$ and conclude that 
   \begin{align*}
   &\lambda_1(t) \lesssim_1 |t| \delta(t)^2; & &\left|\int_t^0 G_-(s) {\rm d} s + \zeta_1 \lambda_1(t)^{1/2}\right| \lesssim_1 |t|^{1/2} \delta(t)^{9/5}. 
 \end{align*}
Here $G_-$ is the (nonlinear) radiation profile of $u$. The fact that $\vec{u}(t)$ is supported in the ball $\{x: |x|<|t|\}$ plays an important role in the argument, as $r= |t|$ becomes a natural reference point when we evaluate the first bubble size. The case of global solutions is much different. Indeed, a similar estimate on the bubble size is not likely to hold. Let us consider the case of ground state $W(x)$. The bubble size is a constant $1$ but the radiation strength is zero, i.e. $G_- = 0$ and $\delta=0$. 
\end{remark}

\section{Choice of good times}

Throughout this section we assume that $u$ is a radial type II blow-up solution without radiation, with (nonlinear) radiation profile $G_-$ in the negative time direction, radiation part $v_{t,L}$, first bubble size $\lambda_1(t)$ and norm function $\delta(t)$, as defined at the beginning of Section 4. Without loss of generality, we also assume that the sign $\zeta_1$ of the first bubble is negative. In this section we show that we may find times $t$ satisfying
\begin{itemize}
 \item[(a)] These times are sufficiently close to the blow-up time $0$;
 \item[(b)] ``Almost best possible'' lower bound estimate on the size of first bubble holds for these times. As shown in Lemma \ref{size of the first bubble lemma}, the inequality $\lambda_1(t) \lesssim_1 |t|\delta(t)^2$ holds for all sufficiently small time. Given any small constant $\kappa$, we show that the roughly reversed inequality $\lambda_1(t) \gtrsim_1 |t| \delta(t)^{2+\kappa}$ also holds for selected times. 
 \item[(c)] Refined soliton resolution holds for these times. Proposition \ref{main tool} implies that the following estimates hold for all sufficiently small times:
 \begin{align*}
  &\frac{\lambda_{j+1}(t)}{\lambda_j(t)} \lesssim_j \delta(t)^2;& & \|\vec{w}(t;\cdot,0)\|_{\mathcal{H}} \lesssim_n \delta(t). 
 \end{align*}
 We apply the classical theory of maximal functions to deduce that the estimate 
 \[
  \tau(t) \doteq \sup_{\lambda \leq \lambda_1(t)} \left\|\chi_0 W_\lambda^4 v_{t,L}\right\|_{L^1 L^2} + \delta(t)^5 \leq \delta(t)^{2-\kappa}
 \]
 also holds for selected times. Proposition \ref{refined main tool} immediately gives the following refined soliton resolution estimates for these times in further argument. 
 \begin{align*}
  &\frac{\lambda_{j+1}(t)}{\lambda_j(t)} \lesssim_j \delta(t)^{4-2\kappa};& & \|\vec{w}(t;\cdot,0)\|_{\mathcal{H}} \lesssim_n \delta(t)^{2-\kappa}. 
 \end{align*}
\end{itemize} 
The estimates given above play an essential role in the proof of our main theorem. We first find times satisfying (a) and (b), and then slightly adjust the values of $t$ to make them satisfy (c) as well. 

\begin{lemma} \label{lemma big lambda one}
 Given any constant $\kappa > 0$ and time $t_0 < 0$, there exists a time $t\in [t_0,0)$, such that $\lambda_1(t) \geq |t| \delta(t)^{2+\kappa}$. 
\end{lemma}
\begin{proof}
 We prove this by a contradiction. Let us assume that $\lambda_1(t) < |t| \delta(t)^{2+\kappa}$ for all $t\in [t_0,0)$. Without loss of generality, we may assume $\delta(t_0) < \min\{\delta_1(n), c\}$, where $\delta_1(n)$ and $c$ are the constants given in Lemma \ref{large energy lemma}. According to Lemma \ref{large energy lemma} and the fact $E(u) = n E(W,0)$, for any $t\in [t_0,0)$, there would exist a nonnegative integer $k(t)$, such that 
 \begin{equation} \label{contradiction hyper}
  \|G_-\|_{L^2(t,t+2^{k(t)} \lambda_1(t))} > c \cdot 2^{k(t)} \delta(t)^2. 
 \end{equation}
 Since $\|G_-\|_{L^2(t,+\infty)} = \delta(t)$, we have 
 \begin{equation} \label{upper bound 2kt}
  c\cdot 2^{k(t)} \delta(t)^2 < \delta(t) \qquad \Longrightarrow \qquad 2^{k(t)} < c^{-1} \delta(t)^{-1}. 
 \end{equation}
 It follows that
 \[
  2^{k(t)} \lambda_1(t) \leq \frac{\delta(t)}{c\delta(t)^2} \cdot |t| \delta(t)^{2+\kappa} = c^{-1} |t| \delta(t)^{1+\kappa} < |t|. 
 \]
 Thus we make define a time sequence $t_0 < t_1 < t_2 < \cdots < 0$ inductively 
 \[ 
  t_{m+1} = t_m + 2^{k(t_m)} \lambda_1(t_m), \qquad m\geq 0.
 \]
 For convenience we also define a decreasing sequence of positive numbers
 \begin{align*}
  \varphi_m = \delta(t_m)^2 = \int_{t_m}^0 |G_-(s)|^2 {\rm d} s. 
 \end{align*}
 By \eqref{contradiction hyper} the sequence $\varphi_m$ satisfies 
 \begin{equation} \label{varphi estimate}
  \varphi_m - \varphi_{m+1} > c^2 2^{2k(t_m)} \delta(t_m)^4 = c^2 2^{2k(t_m)} \varphi_m^2. 
 \end{equation}
 This guarantees that $\varphi_m \rightarrow 0$. We may further choose a sequence $0=m_0<m_1<m_2<\cdots$ by 
 \[
  m_{\ell+1} = \min\{m > m_\ell: \varphi_{m} < \varphi_{m_\ell}/2\}, \qquad \ell \geq 0.
 \]
 This definition, as well as \eqref{varphi estimate}, means that 
 \begin{align*}
  \frac{1}{2} \leq \prod_{m=m_\ell}^{m_{\ell+1}-2} \frac{\varphi_{m+1}}{\varphi_m} \leq  \prod_{m=m_\ell}^{m_{\ell+1}-2} \left(1-c^2 2^{2k(t_m)} \varphi_m\right).
 \end{align*}
 This implies that 
 \begin{equation} \label{ell minus one estimate}
  \sum_{m=m_\ell}^{m_{\ell+1}-2} 2^{2k(t_m)} \varphi_m \lesssim_1 c^{-2}. 
 \end{equation}
 Next we claim that $t_m$ converges to a negative number as $m \rightarrow +\infty$. This is equivalent to saying
 \[
  \prod_{m=0}^\infty \frac{t_{m+1}}{t_m} > 0. 
 \]
 Since 
 \[
  \frac{t_{m+1}}{t_m} = 1 - \frac{2^{k(t_m)}\lambda_1(t_m)}{|t_m|} \geq 1 - 2^{k(t_m)} \delta(t_m)^{2+\kappa}, 
 \]
 with (we recall \eqref{upper bound 2kt} and the assumption $\delta(t_0) < c$)
 \begin{equation} \label{ell estimate}
  2^{k(t_m)} \delta(t_m)^{2+\kappa} \leq c^{-1} \delta(t_m)^{1+\kappa} \leq \varphi_m^{\kappa/2} \ll  1,
 \end{equation}
 it suffice to show 
 \[
  \prod_{m=0}^\infty \left(1 - 2^{k(t_m)} \delta(t_m)^{2+\kappa}\right) > 0,
 \]
 or 
 \[
  \sum_{m=0}^\infty 2^{k(t_m)} \delta(t_m)^{2+\kappa} = \sum_{m=0}^\infty 2^{k(t_m)} \varphi_m^{1+\kappa/2} < +\infty.
 \]
 We may verify this by combining \eqref{ell minus one estimate}, \eqref{ell estimate} and the fact that $\varphi_{m_\ell}$ decreases at least exponentially  
 \begin{align*}
  \sum_{m=0}^\infty 2^{k(t_m)} \varphi_m^{1+\kappa/2} & = \sum_{\ell = 0}^\infty \sum_{m=m_\ell}^{m_{\ell+1}-1} 2^{k(t_m)} \varphi_m^{1+\kappa/2} \\
  & \leq \sum_{\ell = 0}^\infty \left(\varphi_{m_\ell}^{\kappa/2} \sum_{m=m_\ell}^{m_{\ell+1}-2} \left(2^{2k(t_m)} \varphi_m \right) +\varphi_{m_{\ell+1}-1}^{\kappa/2}\right) \\
  & \lesssim_1 \sum_{\ell =0}^\infty c^{-2} \varphi_{m_\ell}^{\kappa/2} < +\infty.
 \end{align*}
 Now we let $t_\infty = \lim t_m < 0$. Then 
 \[
  \int_{t_\infty}^0 |G_-(s)|^2 {\rm d} s \leq \int_{t_m}^0 |G_-(s)|^2 {\rm d} s = \varphi_m \rightarrow 0 \quad \Longrightarrow \quad \int_{t_\infty}^0 |G_-(s)|^2 {\rm d} s = 0. 
 \]
 This gives a contradiction. 
\end{proof}

\begin{lemma} \label{choice of t ast}
 Given any time $t_0 < 0$ and a small constant $\kappa > 0$, there exists a time $t\in [t_0,0)$, such that the radiation profile $G_-$ satisfies 
 \begin{align*}
   \int_{t}^0 G_-(s) {\rm d} s > \frac{1}{2} |t|^{1/2} \delta(t)^{1+\kappa/2}; 
 \end{align*}
 In addition, we also have
 \[
  \tau(t) \doteq \sup_{\lambda\leq \lambda_1(t)} \left\|\chi_0 W_\lambda^4 v_{t,L}\right\|_{L^1 L^2} + \delta(t)^5 \leq \delta(t)^{2-\kappa}. 
 \]
\end{lemma}
\begin{proof}
 Without loss of generality we may always assume that $\delta(t_0)$ is sufficiently small, by simply making $t_0$ approach zero. We may first apply Lemma \ref{lemma big lambda one} and find a time $t_\ast \in [t_0,0)$ such that $\lambda_1(t_\ast) \geq |t_\ast| \delta(t_\ast)^{2+\kappa}$. According to Remark \ref{reference point} and our assumption $\zeta_1 = -1$, we must have that 
 \begin{align*}
   \int_{t_\ast}^0 G_-(s) {\rm d} s \geq \frac{9}{10} |t_\ast|^{1/2} \delta(t_\ast)^{1+\kappa/2}. 
 \end{align*}
 We next choose $t^\ast = t_\ast(1-\delta(t_\ast)^{\kappa}/16) \in (t_\ast,0)$. By the Cauchy-Schwarz inequality, we have 
 \begin{align*}
  \int_{t_\ast}^{t^\ast} |G_-(s)| {\rm d} s \leq (t^\ast-t_\ast)^{1/2} \delta(t_\ast) \leq \frac{1}{4} |t_\ast|^{1/2} \delta(t_\ast)^{1+\kappa/2}. 
 \end{align*}
 Thus for any time $t\in [t_\ast,t^\ast]$, we have 
 \begin{equation} \label{lower bound 0t}
  \int_{t}^0 G_-(s) {\rm d} s > \frac{1}{2} |t_\ast|^{1/2} \delta(t_\ast)^{1+\kappa/2} \geq \frac{1}{2} |t|^{1/2} \delta(t)^{1+\kappa/2}. 
 \end{equation}
 In order to finish the proof, it suffices to pick up a time $t\in [t_\ast,t^\ast]$ such that the second inequality concerning $\tau(t)$ holds. We consider the maximal function 
 \[
  (\mathbf{M} G_-)(s) = \sup_{r>0} \frac{1}{r} \int_s^{s+r} |G_-(s)| {\rm d} s, \qquad s>t_\ast. 
 \]
 The classical theory of maximal function immediately gives 
 \[
  \|\mathbf{M} G_-\|_{L^2([t_\ast,t^\ast])} \lesssim_1 \|G_-\|_{L^2([t_\ast,+\infty))} = \delta(t_\ast). 
 \]
 Therefore there exists a time $t\in [t_\ast, t^\ast]$ such that 
 \[
   (\mathbf{M} G_-)(t) \lesssim_1 (t^\ast-t_\ast)^{-1/2}\delta(t_\ast) \lesssim_1 |t_\ast|^{-1/2} \delta(t_\ast)^{1-\kappa/2}. 
 \]
 We then apply Lemma \ref{first lemma upper bound tau} to deduce (we recall $\lambda_1(t) \lesssim_1 |t|\delta(t)^2$)
 \[
   \sup_{\lambda< \lambda_1(t)} \left\|\chi_0 W_\lambda^4 v_{t,L}\right\|_{L^1 L^2} \lesssim_1 \lambda_1(t)^{1/2} (\mathbf{M} G_-)(t) \lesssim_1 \delta(t)\delta(t_\ast)^{1-\kappa/2}. 
 \]
 Next we recall \eqref{lower bound 0t} and apply the Cauchy-Schwarz to deduce 
 \[
  \frac{1}{2} |t_\ast|^{1/2} \delta(t_\ast)^{1+\kappa/2} < \int_{t}^0 G_-(s) {\rm d} s \leq \delta(t) |t|^{1/2} \quad \Longrightarrow \quad \delta(t_\ast) \lesssim_1 \delta(t)^\frac{1}{1+\kappa/2}.
 \]
 Inserting this upper bound into the estimate above, we obtain 
 \[
  \tau(t) \lesssim_1 \delta(t)^\frac{2}{1+\kappa/2}. 
 \]
 Since $\frac{2}{1+\kappa/2} > 2-\kappa$ is true for any small positive constant $\kappa$, the second inequality
 \[
  \tau(t) \leq \delta(t)^{2-\kappa}
 \]
 holds, as long as $\delta(t_\ast)$ is sufficiently small. 
\end{proof}

\section{The virial identity}

Now let $u$ be a type II blow-up solution with $n$ bubble without radiation, with $T_+ = 0$ and $\zeta_1 = -1$, as described in the previous sections. In the last step of the proof  we apply the virial identity and give a contradiction. We let 
\[
 I(t) = \int_0^\infty u_t(r,t) \left(r u_r(r,t) + \frac{1}{2} u(r,t) \right) r^2 {\rm d} r.   
\]
A direct calculation shows that 
\begin{align*}
 I'(t) & = \int_0^\infty \left[u_{tt} \left(r u_r + \frac{1}{2} u \right) r^2 + u_t \left(r u_{rt} + \frac{1}{2}u_t\right)r^2\right] {\rm d} r \\
 & = \int_0^\infty \left[\left(u_{rr} + \frac{2}{r} u_r + |u|^4 u\right)\left(r u_r + \frac{1}{2} u \right) r^2 - r^2 |u_t|^2 \right] {\rm d} r\\
 & = \int_0^\infty \left[r^3 u_r u_{rr} + \frac{1}{2} r^2 u u_{rr} + 2 r^2 u_r^2 + r u u_r - r^2 |u_t|^2 \right] {\rm d} r\\
 & = - \int_0^\infty |u_t|^2 r^2 {\rm d} r. 
\end{align*}
It is clear that 
\[
 \left|I(t) \right| = \left|\int_0^{|t|} r \cdot u_t(r,t) \left(u_r(r,t) + \frac{1}{2r} u(r,t) \right) r^2 {\rm d} r\right| \lesssim_1 |t| \|\vec{u}(t)\|_{\mathcal{H}}^2. 
\]
Thus we may integrate $I'(t)$ from $t_\ast < 0$ to $t=0$ and obtain 
\begin{equation} \label{virial identity main}
 I(t_\ast) = \int_{t_\ast}^0 \int_0^{\infty} |u_t(r,t)|^2 r^2 {\rm d} r {\rm d} t. 
\end{equation}
We recall that 
\begin{equation} \label{soliton resolution final}
 \vec{u}(\cdot,t) = \sum_{j=1}^n \zeta_j (W_{\lambda_j(t)},0) + \vec{v}_{t,L} (\cdot, 0) + \vec{w}(t; \cdot,0). 
\end{equation} 
According to Proposition \ref{refined main tool}, we have 
\begin{align}
 &\frac{\lambda_{j+1}(t)}{\lambda_j(t)} \lesssim_j \tau(t)^2, \quad j=1,2,\cdots,n-1; & &\|\vec{w}(t;\cdot,0)\|_{\mathcal{H}} \lesssim_n \tau(t), \label{soliton inequality final}
\end{align}
with 
\[
 \tau(t) = \sup_{\lambda\leq \lambda_1(t)} \left\|\chi_0 W_\lambda^4 v_{t,L}\right\|_{L^1 L^2} + \delta(t)^5.  
\]

In order to calculate the contribution of $\vec{v}_{t,L}(\cdot,0)$ in the integrals on both sides of \eqref{virial identity main}, we need the following lemma 
\begin{lemma}\label{explicit vtL}
 Let $v^L$ be a radial free wave with radiation profile $G_-(s)$. If $G_-(s)$ is supported in the interval $[0,R]$, then for $r\in (0,R)$ we have 
 \begin{align*}
  &v^L(r,0) = \frac{1}{r} \int_0^r G_-(s) {\rm d} s; & &v_r^L(r,0) =  \frac{1}{r} G_-(r) - \frac{1}{r^2} \int_0^r G_-(s) {\rm d} s; & & v_t^L(r,0) = \frac{1}{r} G_-(r). 
 \end{align*}
 In addition, we have 
 \begin{align*}
  \int_0^R |v_t^L (r,0)|^2 r^2 {\rm d} r & =  \int_0^R |G_-(s)|^2 {\rm d} s; \\
  \int_0^R v_t^L(r,0) \left(r v_r^L(r,0) + \frac{1}{2} v^L(r,0)\right) r^2 {\rm d} r & = \int_0^R  s |G_-(s)|^2 {\rm d} s - \frac{1}{4} \left(\int_{0}^R G_-(s) {\rm d} s\right)^2. 
 \end{align*}
\end{lemma}
\begin{proof}
 The proof is a straightforward calculation. First of all, the explicit formula of $v^L$ and the assumption on the support of $G_-$ immediately gives 
 \[
  v_L (r,t) = \frac{1}{r} \int_{t-r}^{t+r} G_-(s) {\rm d} s = \frac{1}{r} \int_0^{t+r} G_-(s) {\rm d} s, \qquad r>|t|. 
 \]
 The expressions of $v^L(r,0)$, $v_r^L(r,0)$, $v_t^L(r,0)$ then immediately follows. Now we consider the two integral identities. The first one is trivial. We calculate the second one
 \begin{align*}
  \int_0^R v_t^L \left(r v_r^L + \frac{1}{2} v^L\right) r^2 {\rm d} r & = \int_0^R \frac{G_-(r)}{r} \left(G_-(r) - \frac{1}{2r} \int_0^r G_-(s) {\rm d} s\right) r^2 {\rm d} r\\
  & = \int_0^R  \left(r|G_-(r)|^2 - \frac{1}{2} G_-(r) \int_0^r G_-(s) {\rm d} s\right) {\rm d} r\\
  & = \int_0^R  s|G_-(s)|^2 {\rm d} s - \frac{1}{4} \left(\int_{0}^R G_-(s) {\rm d} s\right)^2. 
 \end{align*}
 This completes the proof. 
\end{proof}

Now let us give an approximation of the integral in the right hand side of \eqref{virial identity main}. We consider the maximal function 
\[
 (\mathbf{M} G_-) (t) = \sup_{r>0} \frac{1}{r} \int_{t}^{t+r} |G_-(s)| {\rm d} s. 
\]
The classical theory of maximal function gives 
\[
 \|\mathbf{M} G_-\|_{L^2(t_\ast,0)} \lesssim_1 \|G_-\|_{L^2(t_\ast,+\infty)} = \delta(t_\ast). 
\]
We recall Lemma \ref{first lemma upper bound tau}, as well as Lemma \ref{size of the first bubble lemma}, and deduce 
\[
 \tau(t) \lesssim_1 \lambda_1(t)^{1/2} (\mathbf{M} G_-) (t) + \delta(t)^5 \lesssim_1 |t_\ast|^{1/2} \delta(t_\ast) (\mathbf{M} G_-) (t) + \delta(t_\ast)^5. 
\]
It follows from the support of $u$, the soliton resolution \eqref{soliton resolution final} and the estimate $\|\vec{w}(t;\cdot,0)\|_{\mathcal{H}} \lesssim_n \tau(t)$ that ($t'$ is the time variable of $v_{t,L}(r,t')$)
\begin{align*}
 \int_0^\infty |u_t(r,t)|^2 r^2 {\rm d} r & = \int_0^{|t|} |u_t(r,t)|^2 r^2 {\rm d} r \\
 & = \int_0^{|t|} \left|\partial_{t'} v_{t,L} (r,0)\right|^2 r^2 {\rm d} r+ \mathcal{E}_1(t). 
\end{align*}
Here the error term satisfies 
\[
 |\mathcal{E}_1(t)| \lesssim_n |t_\ast|^{1/2} \delta(t_\ast)^2 (\mathbf{M} G_-) (t) + |t_\ast| \delta(t_\ast)^2 (\mathbf{M} G_-) (t)^2 + \delta(t_\ast)^6.
\]
Combining this with Lemma \ref{explicit vtL}, we have 
\[
 \int_0^\infty |u_t(r,t)|^2 r^2 {\rm d} r = \int_t^0 |G_-(s)|^2 {\rm d} s + \mathcal{E}_1(t). 
\]
An integration shows that 
\begin{equation} \label{right hand integral main} 
 \int_{t_\ast}^0 \int_0^\infty |u_t(r,t)|^2 r^2 {\rm d} r {\rm d} t = \int_{t_\ast}^0 (s-t_\ast) |G_-(s)|^2 {\rm d} s + \mathcal{E}. 
\end{equation}
Here the error term satisfies 
\begin{align*}
 \left|\mathcal{E} \right| &= \left|\int_{t_\ast}^0 \mathcal{E}_1(t) {\rm d} t \right| \leq \int_{t_\ast}^0 |\mathcal{E}_1(t)| {\rm d} t\\
 & \lesssim_n |t_\ast|^{1/2} \delta(t_\ast)^2 \int_{t_\ast}^0 (\mathbf{M} G_-) (t){\rm d} t + |t_\ast| \delta(t_\ast)^2 \int_{t_\ast}^0 (\mathbf{M} G_-) (t)^2 {\rm d} t + |t_\ast|\delta(t_\ast)^6 \\
 & \lesssim_n |t_\ast| \delta(t_\ast)^3. 
\end{align*}
Next we give an estimate of $I(t_\ast)$. We recall the support and the soliton resolution of $\vec{u}(t_\ast)$ given in \eqref{soliton resolution final} to write 
\begin{align*}
 I(t_\ast)  = \int_0^{|t_\ast|} u_t(r,t_\ast) \left(r u_r(r,t_\ast) + \frac{1}{2} u(r,t_\ast) \right) r^2 {\rm d} r = \sum_{k=1}^6 I_k
\end{align*}
Here (we use $\lambda_j$ instead of $\lambda_j(t_\ast)$ for simplicity)
\begin{align*}
 I_1 & = \int_0^{|t_\ast|} \partial_{t} v_{t_\ast,L}(r,0) \left(r \partial_r v_{t_\ast,L}(r,0) + \frac{1}{2} v_{t_\ast,L}(r,0) \right) r^2 {\rm d} r; \\
 I_2 & = \sum_{j=1}^n \zeta_j \int_0^{|t_\ast|} \partial_{t} v_{t_\ast,L}(r,0) \left(r \partial_r W_{\lambda_j} + \frac{1}{2} W_{\lambda_j}\right) r^2 {\rm d} r; \\
 I_3 & = \int_0^{|t_\ast|} \partial_{t} v_{t_\ast,L}(r,0) \left(r \partial_r w(t_\ast; r,0) + \frac{1}{2} w(t_\ast; r,0) \right) r^2 {\rm d} r; \\
 I_4 & = \int_0^{|t_\ast|} \partial_{t} w(t_\ast, r,0) \left(r \partial_r v_{t_\ast,L}(r,0) + \frac{1}{2} v_{t_\ast,L}(r,0) \right) r^2 {\rm d} r; \\
 I_5 & = \sum_{j=1}^n \zeta_j \int_0^{|t_\ast|} \partial_t w(t_\ast,r,0) \left(r \partial_r W_{\lambda_j} + \frac{1}{2} W_{\lambda_j}\right) r^2 {\rm d} r; \\
 I_6 & = \int_0^{|t_\ast|} \partial_{t} w(t_\ast;r,0) \left(r \partial_r w(t_\ast; r,0) + \frac{1}{2} w(t_\ast; r,0) \right) r^2 {\rm d} r.
\end{align*}
We choose a small constant $\kappa \ll 1$ and let $t_\ast$ be a sufficiently small time satisfying Lemma \ref{choice of t ast}. Inserting the upper bound $\tau(t_\ast) \leq \delta(t_\ast)^{2-\kappa}$ into \eqref{soliton inequality final} immediately gives that 
\begin{align*}
 &\frac{\lambda_{j+1}}{\lambda_j} \lesssim_j \delta(t_\ast)^{4-2\kappa},\; j=1,2,\cdots n-1; & & \|\vec{w}(t_\ast,\cdot,0)\|_{\mathcal{H}} \lesssim_n \delta(t_\ast)^{2-\kappa}. 
\end{align*}
It also implies that 
\[
 \left\|r \partial_r w(t_\ast, r,0) + \frac{1}{2} w(t_\ast; r,0)\right\|_{L^2(\{x: |x|<t_\ast\})} \lesssim_1 |t_\ast| \|\vec{w}(t_\ast; \cdot,0)\|_{\mathcal{H}} \lesssim_n |t_\ast|  \delta(t_\ast)^{2-\kappa}. 
\]
Thus 
\[
 |I_3| +|I_4| + |I_6| \lesssim_n |t_\ast| \delta(t_\ast)^{3-\kappa}. 
\]
In addition, we may apply Lemma \ref{explicit vtL} and obtain 
\[
 I_1 = \int_{t_\ast}^0 (s-t_\ast) |G_-(s)|^2 {\rm d} s - \frac{1}{4} \left(\int_{t_\ast}^0 G_-(s) {\rm d} s\right)^2. 
\]
In order to calculate the other two terms, we first conduct a direct calculation 
\begin{align*}
 r \partial_r W_{\lambda_j} + \frac{1}{2} W_{\lambda_j} & = r \cdot \frac{-r}{\lambda_j^{5/2}} \left(\frac{1}{3} + \frac{r^2}{\lambda_j^2}\right)^{-3/2} +  \frac{1}{2\lambda_j^{1/2}} \left(\frac{1}{3} + \frac{r^2}{\lambda_j^2}\right)^{-1/2}\\
 & = \frac{1}{\lambda_j^{1/2}}\left(\frac{1}{6} - \frac{r^2}{2\lambda_j^2}\right) \left(\frac{1}{3} + \frac{r^2}{\lambda_j^2}\right)^{-3/2}.
\end{align*}
Therefore we have 
\[
 \left\|r \partial_r W_{\lambda_j} + \frac{1}{2} W_{\lambda_j}\right\|_{L^2([0,|t_\ast|]; r^2 {\rm d} r)} \lesssim_1 \lambda_j^{1/2} |t_\ast|^{1/2} \lesssim_1 |t_\ast| \delta(t_\ast). 
\]
As a result, we still have 
\[
 |I_5| \lesssim_n |t_\ast| \delta(t_\ast)^{3-\kappa}. 
\]
Finally we consider $I_2$.
\[
 I_2 = \sum_{j=1}^n \zeta_j \int_0^{|t_\ast|} \partial_{t} v_{t_\ast,L}(r,0) \left(r \partial_r W_{\lambda_j} + \frac{1}{2} W_{\lambda_j}\right) r^2 {\rm d} r = \sum_{j=1}^n I_{2,j}.
\]
The $L^2$ estimate above gives 
\[
 \left|I_{2,j}\right| \lesssim_1 \delta(t_\ast) \lambda_j^{1/2} |t_\ast|^{1/2} \lesssim_n |t_\ast|^{1/2} \delta(t_\ast) \lambda_1^{1/2} (\delta(t_\ast))^{(2-\kappa)(j-1)} \lesssim_n |t_\ast| \delta(t_\ast)^{2+(2-\kappa)(j-1)}.  
\]
Thus 
\[
 \sum_{j=2}^n |I_{2,j}| \lesssim_n |t_\ast| \delta(t)^{4-\kappa}. 
\]
We then calculate $I_{2,1}$ in details 
\begin{align*}
 I_{2,1} &= - \int_0^{|t_\ast|} \frac{G_-(r+t_\ast)}{r} \cdot \frac{1}{\lambda_1^{1/2}}\left(\frac{1}{6} - \frac{r^2}{2\lambda_1^2}\right) \left(\frac{1}{3} + \frac{r^2}{\lambda_1^2}\right)^{-3/2} r^2 {\rm d} r \\
& =  - \lambda_1^{1/2} \int_0^{|t_\ast|} G_-(r+t_\ast) \left[\frac{r}{\lambda_1} \left(\frac{1}{6} - \frac{r^2}{2\lambda_1^2}\right) \left(\frac{1}{3} + \frac{r^2}{\lambda_1^2}\right)^{-3/2}\right]  {\rm d} r. 
\end{align*}
We let $g(y) = y (1/6-y^2/2)(1/3+y^2)^{-3/2}$. It is not difficult to see 
\[
 \left|g(y)+\frac{1}{2}\right| \lesssim_1 \min\{1, y^{-2}\}, \qquad y>0. 
\]
Thus we have 
\begin{align*}
 \left|I_{2,1} - \frac{\lambda_1^{1/2}}{2} \int_{t_\ast}^0 G_-(s) {\rm d} s\right| & \lesssim_1 \lambda_1^{1/2} \int_0^{\lambda_1} |G_-(r+t_\ast)| {\rm d} r + \lambda_1^{5/2} \int_{\lambda_1}^{|t_\ast|} |G_-(r+t_\ast)| r^{-2} {\rm d} r \\
 & \lesssim_1 \lambda_1 \delta(t_\ast) \lesssim_1 |t_\ast| \delta(t_\ast)^3.  
\end{align*} 
In summary, we have 
\begin{align*}
 I(t_\ast) = \int_{t_\ast}^0 (s-t_\ast) |G_-(s)|^2 {\rm d} s - \frac{1}{4} \left(\int_{t_\ast}^0 G_-(s) {\rm d} s\right)^2 + \frac{\lambda_1^{1/2}}{2} \int_{t_\ast}^0 G_-(s) {\rm d} s +  O_n (|t_\ast|\delta(t_\ast)^{3-\kappa})
\end{align*}
Inserting this and \eqref{right hand integral main} into \eqref{virial identity main} yields 
\begin{equation} \label{final contra ineq} 
J\doteq  \left|\frac{\lambda_1^{1/2}}{2} \int_{t_\ast}^0 G_-(s) {\rm d} s - \frac{1}{4} \left(\int_{t_\ast}^0 G_-(s) {\rm d} s\right)^2\right| \lesssim_n |t_\ast| \delta(t_\ast)^{3-\kappa}. 
\end{equation}
Since $t_\ast$ satisfies the conclusion of Lemma \ref{choice of t ast}, we have 
\[
 \int_{t_\ast}^0 G_-(s) {\rm d} s > \frac{1}{2} |t_\ast|^{1/2} \delta(t_\ast)^{1+\kappa/2}. 
\]
In addition, the following inequality holds by Remark \ref{reference point}
   \begin{align*}
 \left|\int_{t_\ast}^0 G_-(s) {\rm d} s - \lambda_1^{1/2}\right| \lesssim_1 |t_\ast|^{1/2} \delta(t_\ast)^{9/5}. 
 \end{align*}
As a result, the left hand side of the inequality \eqref{final contra ineq} satisfies 
\begin{align*}
 J \gtrsim_1 \left(\int_{t_\ast}^0 G_-(s) {\rm d} s\right)^2 \gtrsim_1  |t_\ast| \delta(t_\ast)^{2+\kappa}. 
\end{align*}
This gives a contradiction, since we may always make $\delta(t_\ast)$ sufficiently small. 
\section*{Acknowledgement}
The author is financially supported by National Natural Science Foundation of China Project 12471230.

\end{document}